\documentclass[centertags,twoside,reqno,11pt]{amsart}
\usepackage{amsmath,amstext,amsthm,amscd}
\usepackage{amssymb}
\usepackage[mathscr]{eucal}
\usepackage{mathrsfs}
\usepackage{typearea}
\usepackage{url}
\usepackage{array}
\usepackage[table]{xcolor}
\usepackage{mathabx}
\usepackage{comment}
\usepackage[margin=1in]{geometry} 

\newcommand{\we}{w^{\eps}}
\newcommand{\ze}{z^{\eps}}
\newcommand{\sumi}{\sum_{i=1}^{I}}
\newcommand{\bra}[1]{\left(#1\right)}
\newcommand{\sbra}[1]{\left[#1\right]}
\newcommand{\E}{\mathcal E}
\newcommand{\D}{\mathcal D}

\renewcommand{\mathfrak}{\mathscr}
\renewcommand{\omega}{\varphi}
\renewcommand{\zeta}{\xi}
\newcommand{\ue}{u^{\eps}}
\newcommand{\ve}{v^{\eps}}

\newcommand{\na}{\nabla}
\allowdisplaybreaks
\usepackage[colorlinks=true,citecolor=red,linkcolor=blue]{hyperref}
\usepackage{enumitem} 
\usepackage{color}
\usepackage{tikz}
\usetikzlibrary{arrows,automata,backgrounds,calendar,mindmap,trees}

\numberwithin{equation}{section}
\newtheorem{theorem}{Theorem}[section]
\newtheorem{definition}[theorem]{Definition}

\newtheorem{lemma}[theorem]{Lemma}
\newtheorem{remark}[theorem]{Remark}

\newcommand{\supp}{{\rm supp}}

\newcommand{\eps}{\varepsilon}
\newcommand{\ol}{\overline}
\newcommand{\wh}{\widehat}

\newcommand{\coleq}{\mathrel{\mathop:}=}

\title[Renormalised solutions and equilibration with non-linear diffusion]{Global renormalised solutions and equilibration of reaction--diffusion systems with non-linear diffusion}
\author[K. Fellner]{Klemens Fellner}
\address{Klemens Fellner \hfill\break
	Institute of Mathematics and Scientific Computing, University of Graz, 
	Heinrichstrasse 36, 8010 Graz, Austria}
\email{klemens.fellner@uni-graz.at}
\author[J. Fischer]{Julian Fischer}
\address{Julian Fischer \hfill\break
	Institute of Science and Technology Austria,
	Am Campus 1, 3400 Klosterneuburg, Austria}
\email{julian.fischer@ist.ac.at}
\author[M. Kniely]{Michael Kniely}
\address{Michael Kniely \hfill\break
	Faculty of Mathematics, TU Dortmund University, 
	Vogelpothsweg 87, 44227 Dortmund, Germany}
\email{michael.kniely@tu-dortmund.de}
\author[B. Q. Tang]{Bao Quoc Tang}
\address{Bao Quoc Tang \hfill\break
	Institute of Mathematics and Scientific Computing, University of Graz, 
	Heinrichstrasse 36, 8010 Graz, Austria}
\email{quoc.tang@uni-graz.at, baotangquoc@gmail.com}
\begin{document}
\begin{abstract}
	The global existence of renormalised solutions and convergence to equilibrium for reaction--diffusion systems with non-linear diffusion are investigated. The system is assumed to have quasi-positive non-linearities and to satisfy an entropy inequality. The difficulties in establishing global renormalised solutions caused by possibly degenerate diffusion are overcome by introducing a new class of weighted truncation functions.
	
	By means of the obtained global renormalised solutions, we study the large-time behaviour of complex balanced systems arising from chemical reaction network theory with non-linear diffusion. When the reaction network does not admit boundary equilibria, the complex balanced equilibrium is shown, by using the entropy method, to exponentially attract all renormalised solutions in the same compatibility class. 
	This convergence extends even to a range of non-linear diffusion, where global existence is an open problem, yet we are able to show that solutions to approximate systems converge exponentially to equilibrium uniformly in the regularisation parameter.
\end{abstract}
\maketitle
\tableofcontents
\section{Introduction and main results}
	In this paper, we consider the evolution of concentrations $u = (u_1, \ldots, u_I)$, $I \in \mathbb N$, in a bounded domain $\Omega\subset \mathbb R^d$ with Lipschitz boundary $\partial\Omega$ subject to non-linear diffusion and reactions as modelled by the parabolic system 
	\begin{equation}\label{e0}
		\begin{cases}
		\begin{aligned}
			\partial_tu_i - d_i\Delta u_{i}^{m_i} &= f_i(u) &&\text{ in } Q_T,\\
			\nabla u_i^{m_i}\cdot \nu &= 0 &&\text{ on } \Gamma_T,\\
			u_i(\cdot,0) &= u_{i,0} &&\text{ in } \Omega.
		\end{aligned}
		\end{cases}
	\end{equation}
	Here and below, we set $Q_T \coleq \Omega \times (0, T)$ and $\Gamma_T \coleq \partial \Omega \times (0, T)$ for any $T > 0$, and we employ $\nu$ to denote the unit outward normal vector to $\partial\Omega$. The diffusion constants $d_i$ and the non-linear diffusion exponents $m_i$ are assumed to satisfy $d_i > 0$ and $0 < m_i < 2$. For proving exponential convergence to equilibrium, we further demand $m_i > (d-2)_+/d$ with $(d-2)_+ \coleq \max\{0, d-2\}$. We will comment on the bounds on $m_i$ later on. Moreover, we generically assume $u_{i,0} \in L^1(\Omega)$, $u_{i,0} \geq 0$, and impose the following assumptions on the reaction terms $f_i(u)$, $i=1,\ldots, I$:
	\begin{enumerate}[label=(F\theenumi),ref=F\theenumi]
		\item\label{F1}(Local Lipschitz continuity) $f_i: \mathbb R_+^I \rightarrow \mathbb R$ is locally Lipschitz continuous.
		\item\label{F2} (Quasi-positivity) $f_i(u)\geq 0$ for all $u = (u_1, \ldots, u_I)\in \mathbb R_+^I$ with $u_i = 0$. 
		\item\label{F3} (Entropy inequality) There exist constants $(\mu_i)_{i=1,\ldots, I} \in \mathbb R^I$ satisfying
		\begin{equation}\label{entropy-inequality}	
		\sum_{i=1}^{I}f_i(u)(\log u_i+\mu_i) \leq C \, \sumi \bra{1 + u_i\log u_i}
		\end{equation}
		for all $u\in \mathbb R_+^I$, where $C>0$ is independent of $u$.
	\end{enumerate}
	The second assumption \eqref{F2} guarantees non-negativity of solutions provided that the initial data are non-negative, which is a natural assumption since we consider here $u_i$ as concentrations or densities. 
	Assumption \eqref{F3}, on the other hand, implies existence and control of an \emph{entropy} of the system and it is only slightly stronger than assuming control of the \emph{mass}, i.e.
	\begin{equation}\label{mass-inequality}
		\sum_{i=1}^{I}f_i(u) \leq 0.
	\end{equation}
	
	In general, the assumptions \eqref{F1}, \eqref{F2}, and either \eqref{F3} or \eqref{mass-inequality} of the non-linearities $f_i$ are not enough to obtain suitable a priori estimates which allow to extend local strong solutions globally, see e.g.\ \cite{pierre2000blowup}. On the one hand, for the case of linear diffusion (i.e.\ $m_i=1$ for all $i=1,\ldots, I$), the global existence for \eqref{e0} under extra assumptions on non-linearities has been extensively investigated. For instance, global bounded solutions were shown under \eqref{F3} for the case $d=1$ with cubic non-linearities and $d=2$ with quadratic non-linearities, see \cite{goudon2010regularity,tang2018global}. Recent works \cite{souplet2018global,caputo2019solutions,fellner2020global} showed under either \eqref{F3} or \eqref{mass-inequality} that bounded solutions can be obtained for (slightly super-)quadratic non-linearities in \textit{all dimensions}. We refer the reader to the extensive review \cite{pierre2010global} for related results concerning global existence of bounded or weak solutions. It is remarked that in all such results, additional assumptions on non-linearities must be imposed. A breakthrough has been made in \cite{Fis15} where global {\it renormalised solutions} were obtained without any extra assumptions on the non-linearities. The notion of renormalised solution was initially introduced for Boltzmann's equation in \cite{DL89} and it has been applied subsequently to many other problems; see e.g.\ \cite{DMMOP99,DFPV07,CJ19}. The case of non-linear, possibly degenerate diffusion is, on the other hand, much less understood. Up to our knowledge, there are only two works (\cite{LP17} and \cite{FLT20}) which showed global existence of weak or bounded solutions for porous medium type of diffusion, i.e. $m_i>1$, under \eqref{mass-inequality} and some restricted growth conditions on the non-linearities. A recent preprint, however, provides renormalised solutions to a system as in \eqref{e0} but imposing more regularity on the initial data and the underlying domain \cite{LW21}. Instead of an entropy condition as in \eqref{entropy-inequality}, the authors of \cite{LW21} assume an additional bound on cross-absorptive reaction terms. Furthermore, \cite{FLT20} proved the convergence to equilibrium for \eqref{e0} modelling a single reversible chemical reaction.
	
	\medskip
	The present study shows the existence of global renormalised solutions to \eqref{e0} featuring non-linear diffusion of porous medium type or fast diffusion \textit{without any extra conditions on the non-linearities} (up to assuming \eqref{F1}--\eqref{F3}). In addition, we show that these solutions, in case \eqref{e0} models general complex balanced chemical reaction networks, converge exponentially to equilibrium with explicitly computable rates. Our paper seems to be the first extensive contribution to the non-linear diffusion type system \eqref{e0} establishing---for specific parameter regimes---the existence of global renormalised (or even weak and bounded) solutions as well as exponential convergence to equilibrium in various $L^p$ spaces. A brief summary of our results is given in Table \ref{tablesummarisation} at the end of this section.

	\medskip
	{The first part of this paper} is concerned with the global existence of renormalised solutions to \eqref{e0}. We consider the following definition of renormalised solutions.
	\begin{definition}[Renormalised solutions]\label{def-renormalised}
		Let $m_i < 2$ for all $1 \leq i \leq I$ and $u_0 = (u_{i,0})_i \in L^1(\Omega)^I$ be (componentwise) non-negative. A non-negative function $u = (u_1, \ldots, u_I): \Omega\times (0,\infty) \to \mathbb R_+^I$ is called a renormalised solution to \eqref{e0} with initial data $u_0$ if $u_i\in L^\infty_{loc}(0, \infty; L^1(\Omega))$ and $\nabla {{u_i}^{\frac{m_i}{2}}}\in L^2_{loc}(0, \infty; L^2(\Omega)^I)$, and if for any smooth function $\xi: \mathbb R_+^I \to \mathbb R$ with compactly supported derivative $D\xi$ and any function $\psi \in C^{\infty}(\overline{\Omega} \times (0,\infty))$, we have
		\begin{equation}\label{renormalised-form}
			\begin{aligned}
				&\int_{\Omega}\xi(u(\cdot, T))\psi(\cdot, T) \, dx - \int_{\Omega}\xi(u_0)\psi(\cdot, 0) \, dx - \int_{Q_T}\xi(u)\partial_t\psi \, dx \, dt \\
				&\quad = -\sum_{i,j=1}^{I} d_i m_i \int_{Q_T}\psi\partial_i\partial_j\xi(u)u_i^{m_i-1}\nabla u_i\cdot \nabla u_j \, dx \, dt \\
				&\qquad - \sum_{i=1}^{I} d_i m_i \int_{Q_T} \partial_i\xi(u) u_i^{m_i-1}\nabla u_i \cdot \nabla \psi \, dx \, dt \\
				&\qquad + \sum_{i=1}^{I}\int_{Q_T}\partial_i\xi(u)f_i(u)\psi \, dx \, dt 
			\end{aligned}
		\end{equation}
		for a.e.\ $T>0$.
	\end{definition}
	We notice that all integrals in \eqref{renormalised-form} are well-defined due to the compact support of $D\xi$ and as the integrals relating to gradients of solutions can be rewritten as	
	\begin{equation*}
		\int_{Q_T}\psi\partial_i\partial_j\xi(u)u_i^{m_i-1}\nabla u_i\cdot \nabla u_j \, dx \, dt  = \frac{4}{m_im_j}\int_{Q_T}\psi \partial_i\partial_j\xi(u)u_i^{\frac{m_i}{2}}u_j^{1-\frac{m_j}{2}}\nabla u_i^{\frac{m_i}{2}} \cdot \nabla u_j^{\frac{m_j}{2}} \, dx \, dt 
	\end{equation*}
	and
	\begin{equation*}
		\int_{Q_T}\partial_i\xi(u)u_i^{m_i-1}\nabla u_i\nabla \psi \, dx \, dt  = \frac{2}{m_i} \int_{Q_T}\partial_i\xi(u)u_i^{\frac{m_i}{2}}\nabla  u_i^{\frac{m_i}{2}} \nabla \psi \, dx \, dt 
	\end{equation*}
	which are both finite thanks to the boundedness of $\nabla u_i^{\frac{m_i}{2}}$ in $L^2(Q_T)$ and the restriction $m_i < 2$ for all $1 \leq i \leq I$.

	\medskip	
	The first main result of this article is the following theorem.
	
	\begin{theorem}[Global existence of renormalised solutions]\label{existence-renormalised}
		Let $\Omega\subset \mathbb R^d$ be a bounded domain with Lipschitz boundary $\partial\Omega$. Assume $m_i<2$ for all $i=1,\ldots, I,$ and that the conditions \eqref{F1}--\eqref{F3} hold. Then, for any non-negative initial data $u_0 = (u_{i,0})_i \in L^1(\Omega)^I$ subject to $$\sum_{i=1}^{I}\int_{\Omega}u_{i,0}\log u_{i,0} \, dx < \infty,$$ there exists a global renormalised solution to \eqref{e0}. 
		
		\medskip
		Moreover, any renormalised solution satisfies the following weak entropy--entropy dissipation relation for a.e.\ $0\leq s < T$:
		\begin{equation}\label{entropy_law}
		\begin{aligned}
			\sum_{i=1}^I\int_{\Omega}u_i(\log u_i +\mu_i - 1) \, dx\bigg|_s^T \leq &-\sum_{i=1}^I\frac{4d_i}{m_i^2}\int_s^T\int_\Omega |\nabla (u_i)^\frac{m_i}{2}|^2 \, dx \, dt \\
			&+ \sum_{i=1}^I\int_s^T\int_{\Omega}f_i(u)(\log u_i + \mu_i) \, dx \, dt \leq 0.
		\end{aligned}
		\end{equation}
		
		\medskip
		In addition, if there exist some numbers $q_i \in \mathbb R$ such that $\sum_{i=1}^I q_if_i(u) = 0$, then all renormalised solutions satisfy
		\begin{equation}\label{conservation_law}
			\sum_{i=1}^I \int_{\Omega} q_iu_i(x,t) \, dx = \sum_{i=1}^I \int_{\Omega} q_iu_{i,0}(x) \, dx
		\end{equation}
		for a.e.\ $t>0$.
	\end{theorem}
	
	\begin{remark}\hfill
	\begin{itemize}
		\item The relation \eqref{conservation_law} usually corresponds to mass conservation laws in chemical reactions, see e.g.\ \cite{FT18}.
		\item The uniqueness of renormalised solutions is widely open. For a weaker notion called weak-strong uniqueness, i.e. the renormalised solution is unique as long as a strong solution exists, we refer the interested reader to \cite{Fis17}.
		\item The fact that \eqref{entropy_law} and \eqref{conservation_law} are satisfied by {\normalfont any} renormalised solution plays an important role in the next part of this paper where we show that {\normalfont all} renormalised solutions, instead of only some of them, converge to equilibrium.
	\end{itemize}
	\end{remark}

	The strategy for constructing a global renormalised solution follows the ideas in \cite{Fis15}. Given a sequence of approximate solutions $u^\varepsilon$, one proves compactness of the family $u^\varepsilon$ by deriving bounds on truncations $\varphi_i^E(u^\varepsilon)$, $E \in \mathbb N$, $i = 1, \dotsc, I$, and subsequently employing an Aubin--Lions lemma. The smooth functions $\varphi_i^E$ defined in \eqref{omega} essentially truncate the mappings $u \mapsto u_i$ if $u$ becomes too large (measured in terms of $E$) while leaving $u \mapsto u_i$ unchanged for sufficiently small $u$. The main difficulty in our current situation is caused by the different diffusion exponents $m_i$ for each species $u_i$, which makes the truncations $\varphi_i^E$ provided in \cite{Fis15} not applicable. We overcome this issue by modifying the functions $\varphi_i^E$ in such a way that the truncation is not decided by $\sum_j u_j$ but by a weighted sum $\sum_j E^{-\alpha^i_j} u_j$ with suitable constants $\alpha^i_j \geq 1$ chosen depending on the diffusion exponents $m_i$. We then pass to the limit $\varepsilon \rightarrow 0$ in the equation for $\varphi_i^E(u^\varepsilon)$ at the cost of an additional defect measure which, nevertheless, vanishes in the subsequent limit $E \rightarrow \infty$. Finally, an equation for $\xi(\varphi_i^E(u))$ is derived, where $\xi$ is subject to the same assumptions as in Definition \ref{def-renormalised}, and the limit $E \rightarrow \infty$ is performed resulting in the desired equation \eqref{renormalised-form} for $\xi(u)$.

	\medskip
	The restriction $m_i<2$ for $i=1,\ldots, I$ appears already in the definition of renormalised solutions in Definition \ref{def-renormalised}. It is also crucial in showing the vanishing of the defect measure when establishing the renormalised solution. We do not know whether this restriction is purely technical or due to a deeper reason. It is remarked that the same condition was crucially imposed in \cite{LP17}, where global weak solutions to \eqref{e0} were investigated.
	
	\medskip
	In the second part of this paper, we study the large-time behaviour of reaction--diffusion systems of type \eqref{e0} modelling chemical reaction networks subject to the complex balance condition. The main vocabulary is introduced below but we refer to \cite{DFT17} for a more detailed discussion of the involved concepts. We assume that there are $I$ chemicals $S_1, \ldots, S_I$ reacting via the following $R$ reactions
	\begin{equation}\label{reaction1}
		y_{r,1}S_1 + \cdots + y_{r,I}S_I \xrightarrow{k_r} y_{r,1}'S_1 + \cdots + y_{r,I}' S_I \quad \text{ for all } \quad r = 1,\ldots, R
	\end{equation}
	where $y_{r,i}, y_{r,i}' \in \{0\}\cup [1,\infty)$ are stoichiometric coefficients, and $k_r>0$ are reaction rate constants. Utilizing the notation $y_r = (y_{r,i})_{i=1,\ldots, I}$ and $y_r' = (y_{r,i}')_{i=1,\ldots, I}$, we can rewrite the reactions in \eqref{reaction1} as
	\begin{equation}\label{reaction2}
		y_r \xrightarrow{k_r} y_r' \quad \text{ for all } \quad r=1,\ldots, R.
	\end{equation}
	
	\medskip	
	Denote by $u_i(x,t)$ the concentration of $S_i$ at position $x\in\Omega$ and time $t\geq 0$. The reaction network \eqref{reaction1} results in the following reaction--diffusion system with non-linear diffusion for $u = (u_1, \ldots, u_I)$:
	\begin{equation}\label{mass-action-system}
		\begin{cases}
		\begin{aligned}
			\partial_t u_i - d_i\Delta u_i^{m_i} &= f_i(u) &&\text{ in } Q_T,\\
			\nabla u_i^{m_i}\cdot \nu &= 0 &&\text{ on } \Gamma_T,\\
			u_i(x,0) &= u_{i,0}(x) &&\text{ in } \Omega,
		\end{aligned}
		\end{cases}
	\end{equation}
	in which the reaction term $f_i(u)$ is determined using the \emph{mass action law},
	\begin{equation}\label{reaction-term}
		f_i(u) = \sum_{r=1}^{R}k_ru^{y_r}(y'_{r,i} - y_{r,i}) \quad \text{ where } \quad u^{y_r} = \prod_{i=1}^{I}u_i^{y_{r,i}}.
	\end{equation}
	It is obvious that $f_i(u)$ is locally Lipschitz continuous, hence \eqref{F1} is satisfied. For \eqref{F2}, it is easy to check that $f_i(u) \geq 0$ holds true for all $u\in \mathbb R_{+}^I$ satisfying $u_i = 0$ since we assume $k_r>0$ and $y\in \mathbb (\{0\}\cup [1,\infty))^I$ for all $y\in \{y_r, y_r' \}_{r=1,\ldots, R}$. Before verifying \eqref{F3}, we recall the definition of a \emph{complex balanced equilibrium}. A constant state $u_\infty = (u_{i,\infty})_i \in [0,\infty)^I$ is called a complex balanced equilibrium for \eqref{mass-action-system}--\eqref{reaction-term} if
	\begin{equation}\label{cbe}
		 \sum_{\{r: y_r=y\}}k_ru_{\infty}^{y_r} = \sum_{\{s: y_s' = y\}}k_su_{\infty}^{y_s} \quad \text{ for all } \quad y\in \{y_r, y_r' \}_{r=1,\ldots, R}.
	\end{equation}
	Intuitively, this condition means that for any complex the total outflow from the complex and the total inflow into the complex are balanced at such an equilibrium. Using complex balanced equilibria, one can show that the non-linearities in \eqref{reaction-term} satisfy \eqref{F3} by choosing $\mu_i = -\log u_{i,\infty}$ with a strictly positive complex balanced equilibrium $u_\infty$ (see \cite[Proposition 2.1]{DFT17}). Therefore, one can apply Theorem \ref{existence-renormalised} to obtain global renormalised solutions to \eqref{mass-action-system}--\eqref{reaction-term}.
	
	\medskip
	In general, there might exist infinitely many solutions to \eqref{cbe}. To uniquely identify a \textit{positive} equilibrium, we need a set of conservation laws. More precisely, we consider the Wegscheider matrix (or stoichiometric coefficient matrix)
	\begin{equation*}
		W = \big[ (y'_r - y_r)_{r=1,\ldots, R} \big]^T \in \mathbb R^{R\times I},
	\end{equation*}
	and define $m = \dim(\ker(W))$. If $m>0$, then we can choose a matrix $\mathbb Q \in \mathbb R^{m\times I}$, where the rows form a basis of $\ker(W)$. Note that this implies $\mathbb Q\, (y_r' - y_r) = 0$ for all $r=1,\ldots, R$. Therefore, by recalling $f(u) = (f_i(u))_{i=1,\ldots, I} = \sum_{r=1}^{R}k_ru^{y_r}(y_r' - y_r)$ we obtain
	\begin{equation*}
		\mathbb Q\, f(u) = \sum_{r=1}^{R}k_ru^{y_r}\mathbb Q(y_r' - y_r) = 0 \quad \text{ for any } \quad u\in \mathbb R^I.
	\end{equation*}
	By using the homogeneous Neumann boundary condition, we can then (formally) compute that
	\begin{equation*}
		\frac{d}{dt}\int_{\Omega}\mathbb Q\,u(t) \, dx = \int_{\Omega}\mathbb Q\,f(u) \, dx = 0
	\end{equation*}
	and, consequently, for a.e.\ $t > 0$,
	\begin{equation*}
		\mathbb Q\int_{\Omega}u(t) \, dx = \mathbb Q\int_{\Omega}u_0 \, dx.
	\end{equation*}
	In other words, system \eqref{mass-action-system}--\eqref{reaction-term} possesses $m$ linearly independent mass conservation laws. By the vocabulary used in the literature on chemical reaction network theory, for any fixed non-negative mass vector $M\in \mathbb R^m$, the set $\{u\in \mathbb R^I_+: \mathbb Q\,u=M\}$ is called the {\it compatibility class} corresponding to $M$. In case that for each strictly positive vector $M \in \mathbb R^m$, there exists a strictly positive complex balanced equilibrium, one says that the chemical reaction network is complex balanced. This terminology is well-defined since it was proved in \cite{Hor73} that if one equilibrium is complex balanced, then all equilibria are complex balanced. It was further shown that for each non-negative initial mass $M\in \mathbb R^m$ with $M \neq 0$, there exists a unique strictly positive complex balanced equilibrium $u_{\infty} = (u_{i,\infty})_i \in \mathbb R^I_+$ in the compatibility class corresponding to $M$ (cf.\ \cite{HJ72,Fei79}). 
	We stress that additionally there might exist many so-called {\it boundary equilibria}, which are complex balanced equilibria lying on $\partial \mathbb R_+^I$. It is remarked that there exists a large class of complex balanced chemical reaction networks, called concordant networks, which do not have boundary equilibria, see \cite{shinar2013concordant}. For the sake of brevity, from now on we call the unique strictly positive complex balanced equilibrium simply the {\it complex balanced equilibrium}. Note that all the previous considerations---although established for ODE models for chemical reaction networks---also apply to our PDE setting thanks to the homogeneous Neumann boundary conditions, which allow for spatially homogeneous equilibria.
	
	\begin{theorem}[Exponential equilibration of renormalised solutions]\label{thm2}
		Let $\Omega\subset \mathbb R^d$ be a bounded domain with Lipschitz boundary $\partial\Omega$. Assume that the reaction network \eqref{reaction2} is complex balanced, suppose that \eqref{F1}--\eqref{F3} hold, and let any non-negative initial data $u_0 = (u_{i,0})_i \in L^1(\Omega)^I$ with $\sum_{i=1}^I\int_{\Omega}u_{i,0}\log u_{i,0} \, dx < \infty$ be given.
		\begin{itemize}
		\item If $m_i<2$ for all $1 \leq i \leq I$, there exists a global renormalised solution to \eqref{mass-action-system}--\eqref{reaction-term}. 
		\item If $\frac{(d-2)_+}{d}<m_i$ for all $1 \leq i \leq I$ and assuming that there exist no boundary equilibria, any renormalised solution to \eqref{mass-action-system}--\eqref{reaction-term} converges exponentially to the positive equilibrium $u_\infty \in \mathbb R^I_+$ in the same compatibility class as $u_0$ with a rate which can be explicitly computed up to a finite dimensional inequality, i.e.
		\begin{equation*}
			\sum_{i=1}^I\|u_i(t) - u_{i,\infty}\|_{L^1(\Omega)} \leq Ce^{-\lambda t} \quad \text{ for all } \quad t\geq 0,
		\end{equation*} 
		for some positive constants $C, \lambda > 0$.
		\end{itemize}
	\end{theorem}

	We emphasise that the convergence to equilibrium in Theorem \ref{thm2} is proved for {\it all renormalised solutions}, not only the ones obtained via an approximation procedure. This is possible because the proof of the convergence uses only the weak entropy--entropy dissipation law \eqref{entropy_law} and the conservation laws \eqref{conservation_law} which are satisfied by all renormalised solutions. It is noted that, when a system possesses boundary equilibria, the convergence to the positive equilibrium (more precisely, the instability of boundary equilibria) is very subtle and strongly connected to the famous \textit{Global Attractor Conjecture}, see for instance \cite{DFT17,FT18,craciun2015toric}. 

	\medskip
	Theorem \ref{thm2}, up to our knowledge, is the first result of trend to equilibrium for general complex balanced reaction networks with non-linear diffusion. A special case of a single reversible reaction was considered recently in \cite{FLT20} in which the authors utilised a so-called {\it indirect diffusion effect} (see e.g.\ \cite{einav2020indirect}), which seems difficult to be generalised to general systems. The proof of Theorem \ref{thm2} uses the relative entropy
	\begin{equation}\label{re}
		\E[u|u_\infty] = \sumi \int_{\Omega}\bra{u_i\log\frac{u_i}{u_{i,\infty}} - u_i+u_{i,\infty}}dx
	\end{equation}
	and its corresponding entropy dissipation
	\begin{equation}\label{ed}
		\D[u] = \sumi d_im_i\int_{\Omega}u_i^{m_i-2}|\na u_i|^2 \, dx + \sum_{r=1}^{R}k_ru_{\infty}^{y_r}\int_{\Omega}\Psi\bra{\frac{u^{y_r}}{u_{\infty}^{y_r}}; \frac{u^{y_r'}}{u_{\infty}^{y_r'}}}dx
	\end{equation}
	where $\Psi(x;y) = x\log(x/y) - x + y$. It holds, at least formally, that
	\begin{equation*}
		-\frac{d}{dt}\E[u|u_\infty] = \D[u]
	\end{equation*}
	along any trajectory of \eqref{mass-action-system}--\eqref{reaction-term}. The degeneracy of the non-linear diffusion makes the classical logarithmic Sobolev inequality not applicable. Our main idea here is to utilise some generalised logarithmic Sobolev inequalities (see Lemmas \ref{Lp-LSI} and \ref{Generalised-LSI}), which are suited for non-linear diffusion, and the established results in \cite{FT18} for the case of linear diffusion, to firstly derive an entropy--entropy dissipation inequality of the form
	\begin{equation}\label{ee1}
		\D[u]\geq C(\E[u|u_\infty])^{\alpha}
	\end{equation}
	where $\alpha = \max_{i=1,\ldots, I}\{1,m_i\}$. Note that this functional inequality is proved for all non-negative functions $u:\Omega \to \mathbb R_+^I$ satisfying the conservation laws $\mathbb Q\int_{\Omega}u \, dx = |\Omega|\mathbb Q u_\infty$, and therefore is suitable for renormalised solutions, which have very low regularity. If $\alpha = 1$ in \eqref{ee1}, one immediately gets exponential convergence of the relative entropy to zero and, consequently, exponential convergence of the solutions to equilibrium in $L^1$ thanks to a Csisz\'ar--Kullback--Pinsker type inequality. If $\alpha > 1$, we first obtain an algebraic decay of the relative entropy to zero. Thanks to this, we can explicitly compute a finite time $T_0>0$ from which onwards (in time) the averages of concentrations are strictly bounded below by a positive constant. This helps to compensate the degeneracy of the diffusion and, therefore, to show that solutions with such lower bounds satisfy the linear entropy--entropy dissipation inequality, i.e.
	\begin{equation*}
		\D[u(t)] \geq C\E[u(t)|u_\infty] \quad \text{ for all } \quad t\geq T_0,
	\end{equation*}
	which recovers exponential convergence to equilibrium.
	
	\medskip
	The use of renormalised solutions allows to deal with a large class of non-linearities, but on the other hand it restricts Theorem \ref{thm2} to the case $\frac{(d-2)_+}{d} < m_i < 2$. When the non-linearities are of polynomial type, it was shown in \cite{LP17,FLT20}, under the assumption of the mass dissipation \eqref{mass-inequality} instead of the entropy inequality \eqref{entropy-inequality}, that one can get global weak or even bounded solutions if the porous medium exponents $m_i$ are large enough. In the following theorem, we show an analog result under the entropy inequality condition \eqref{entropy-inequality}. Moreover, the solution is shown to converge exponentially to the positive complex balanced equilibrium.
	\begin{theorem}[Existence and convergence of weak and bounded solutions]\label{thm3}
		Let $\Omega\subset \mathbb R^d$ be a bounded domain with Lipschitz boundary $\partial\Omega$. Assume that the assumptions \eqref{F1}--\eqref{F3} hold, and the non-linearities are bounded by polynomials, i.e.\ there exist $C>0$ and $\mu \geq 1$ such that
		\begin{equation*}
			|f_i(u)| \leq C\bra{1+|u|^{\mu}} \quad \text{ for all } \quad u\in \mathbb R^I, \ i=1,\ldots, I.
		\end{equation*}
		\begin{itemize}
		\item If
		\begin{equation}\label{high_mi}
			m_i \geq \max\{\mu-1, 1\} \quad \text{ for all } \quad i=1,\ldots, I,
		\end{equation}
		then the system \eqref{e0} has a global non-negative weak solution for each non-negative initial data $u_0 = (u_{i,0})_{i=1,\ldots, I}$ such that $u_{i,0}\log u_{i,0} \in L^2(\Omega)$ for all $i=1,\ldots, I$. 
		\item Under an even stronger condition on the porous medium exponents, namely
		\begin{align}\label{stronger_porous}
		\begin{split}
		\min_{i=1,\ldots, I} m_i &> \max\{\mu-1, 1\} \quad \text{ for } \quad d \le 2, \\ 
		\min_{i=1,\ldots, I} m_i &> \mu - \frac{4}{d+2} \quad \text{ for } \quad d\geq 3,
		\end{split}
		\end{align}
		the weak solution to \eqref{e0} subject to bounded non-negative initial data $u_0\in L^\infty(\Omega)^I$ is locally bounded in time, i.e.
		\begin{equation*}
		\|u_i\|_{L^\infty(\Omega \times(0,T))} \leq C_T \quad \text{ for all } \quad i=1,\ldots, I,
		\end{equation*}
		where $C_T$ is a constant growing at most polynomially in $T>0$.
		\end{itemize}
		\medskip
		Assume now that the reaction network \eqref{reaction2} is complex balanced and define
		\begin{equation}\label{mu}
			\mu = \max_{r=1,\ldots, R}\{|y_r|, |y_r'|\}
		\end{equation}
		where $|y_{r}| = \sum_{i=1}^Iy_{r,i}$. Assume moreover that \eqref{reaction2} has no boundary equilibria. 
		\begin{itemize}
			\item
		If \eqref{high_mi} holds and $0\leq u_0$ such that $u_{i,0}\log{u_{i,0}} \in L^2(\Omega)$ for all $i=1,\ldots, I$, all weak solutions to \eqref{mass-action-system}--\eqref{reaction-term} converge exponentially to the complex balanced equilibrium in $L^1(\Omega)$, i.e.
		\begin{equation*}
			\sumi \|u_i(t) - u_{i,\infty}\|_{L^1(\Omega)} \leq Ce^{-\lambda t} \quad \text{ for all } \quad t \geq 0,
		\end{equation*}
		for some $C, \lambda > 0$. 
		
		\item If \eqref{stronger_porous} holds and $0\leq u_0\in L^\infty(\Omega)^I$, weak solutions to \eqref{mass-action-system}--\eqref{reaction-term} are bounded locally in time and converge exponentially to equilibrium in any $L^p$ norm with $p \in [1,\infty)$, i.e.
		\begin{equation}\label{Lp-convergence}
			\sum_{i=1}^I\|u_i(t) - u_{i,\infty}\|_{L^p(\Omega)} \leq C_pe^{-\lambda_p t} \quad \text{ for all } \quad t \geq 0,
		\end{equation}
		with positive constants $C_p, \lambda_p>0$.
		\end{itemize}
	\end{theorem}
	
	
	\medskip
	As one can see from Theorems \ref{thm2} and \ref{thm3}, the global existence and trend to equilibrium for \eqref{mass-action-system}--\eqref{reaction-term} are well-established for either $m_i < 2$ for all $i=1,\ldots, I$, or $m_i$ large enough in the sense of \eqref{high_mi} or \eqref{stronger_porous} for all $i=1,\ldots, I$. There exists, therefore, a gap in which the global existence of any kind of solution remains open.  Remarkably, the proof of the convergence to equilibrium in this paper does not rely on the restriction $m_i < 2$, and it therefore is applicable to any global solution to \eqref{mass-action-system}--\eqref{reaction-term} as long as it satisfies the entropy law \eqref{entropy_law} and the mass conservation laws \eqref{conservation_law}. We, thus, arrive at the following result on the convergence to equilibrium for approximating systems of \eqref{mass-action-system}--\eqref{reaction-term} uniformly in the approximation parameter.
	
	\begin{theorem}[Uniform convergence of approximate solutions]\label{thm4}
		Let $\Omega\subset \mathbb R^d$ be a bounded domain with Lipschitz boundary $\partial\Omega$. Assume that \eqref{mass-action-system}--\eqref{reaction-term} subject to \eqref{F1}--\eqref{F3}, non-negative initial data $u_0 = (u_{i,0})_i \in L^1(\Omega)^I$ with $\sum_{i=1}^I\int_{\Omega}u_{i,0}\log u_{i,0} \, dx < \infty$, and $m_i > \frac{(d-2)_+}{d}$ for all $i = 1, \dotsc, I$ admits a complex balanced equilibrium but no boundary equilibria. For any $\varepsilon>0$, let $u^\varepsilon = (\ue_i)_{i=1,\ldots, I}$ be the solution to the approximate system
		\begin{equation}\label{approx}
		\begin{cases}
		\begin{aligned}
			\partial_tu_{i}^{\varepsilon} - d_i\Delta (u_i^\varepsilon)^{m_i} &= \frac{f_i(u^\varepsilon)}{1+\varepsilon|f(u^\varepsilon)|} &&\text{ in } Q_T,\\
			\nabla (u_i^\varepsilon)^{m_i}\cdot \nu &= 0 &&\text{ on } \Gamma_T,\\
			u_{i}^\varepsilon(x,0) &= u_{i,0}^{\eps}(x) &&\text{ in } \Omega,
		\end{aligned}
		\end{cases}
		\end{equation}
		where $u_{i,0}^\eps \in L^\infty(\Omega)$ satisfies 
		\begin{equation}\label{assump1}
			\max_{i}\lim_{\eps\to 0}\|u_{i,0}^{\eps} - u_{i,0}\|_{L^1(\Omega)} = 0 \quad \text{ and } \quad \max_{i}\sup_{\eps>0}\int_{\Omega}u_{i,0}^\eps \log u_{i,0}^\eps \, dx < \infty.
		\end{equation}
		Moreover, the approximated initial data $u_0^\eps$ are chosen such that, for all $\eps>0$,
		\begin{equation}\label{conservation_law_approx}
			\mathbb Q \int_{\Omega}u_0^{\eps} \, dx = \mathbb Q\int_{\Omega} u_0 \, dx.
		\end{equation}
		(This implies that the system \eqref{approx} admits a unique positive complex balanced equilibrium $u_\infty \in \mathbb R_+^I$, which is independent of $\eps>0$.) 
		
		\medskip
		Then, $u^\eps$ converges exponentially to $u_\infty$ with a rate which is {\normalfont uniform} in $\varepsilon$, i.e.
		\begin{equation*}
			\sum_{i=1}^{I}\|u_{i}^{\eps}(t) - u_{i,\infty}\|_{L^1(\Omega)} \leq Ce^{-\lambda t} \quad \text{ for all } \quad t \geq 0,
		\end{equation*}
		where $C, \lambda>0$ are constants {\normalfont independent} of $\eps>0$.
	\end{theorem}
	To prove Theorem \ref{thm4}, we make use of the same relative entropy as in \eqref{re},
	\begin{equation*}
	\E[u^\eps|u_\infty] = \sum_{i=1}^{I}\int_{\Omega}\bra{u_i^\eps\log\frac{u_i^\eps}{u_{i,\infty}} - u_i^\eps + u_{i,\infty}}dx,
	\end{equation*}
	where $u^\eps$ is the solution to \eqref{approx}. The corresponding non-negative entropy dissipation can be calculated as
	\begin{equation*}
	\D[u^\eps] = \sum_{i=1}^{I}d_i \int_{\Omega}m_i(u_i^\eps)^{m_i-2}|\nabla u_i|^2 \, dx + \sum_{r=1}^Rk_ru_\infty^{y_r} \int_{\Omega}\frac{1}{1+\eps|f(u^\eps)|}\Psi\bigg( \frac{(u^\eps)^{y_r}}{u_\infty^{y_r}}; \frac{(u^\eps)^{y_r'}}{u_\infty^{y_r'}} \bigg) dx.
	\end{equation*}
	Since there is no uniform-in-$\eps$ $L^\infty$ bound for $u^\eps$ available, the factor $\frac{1}{1+\eps|f(\ue)|}$ is not bounded below uniformly in $\eps>0$. Therefore, it seems to be impossible to use the entropy dissipation (for limit solutions) in \eqref{ed} as a lower bound for $\D[\ue]$. We overcome this issue by using the ideas in \cite{FT18}. Roughly speaking, we deal with the second sum in $\D[\ue]$ by estimating it below by a term involving only spatial averages of $\ue$ (rather than $\ue$ pointwise as above), and then exploiting the fact that these averages are bounded uniformly in $\eps>0$.

	\medskip
	We summarise the global existence and convergence to equilibrium for the mass action system \eqref{mass-action-system}--\eqref{reaction-term} in Table \ref{tablesummarisation} (assuming the existence of a unique strictly positive complex balanced equilibrium and the absence of boundary equilibria). Note that the global existence of any kind of solution remains open for diffusion exponents satisfying $2 \leq m_i < \mu - 1$, where $\mu$ is defined in \eqref{mu}.
	
	\begin{table}[h]
	\begin{tabular}{| l || l | l |}
		\hline
		\textbf{Diffusion exponents} & \textbf{Global existence} & \textbf{Convergence to \hfill \linebreak equilibrium} \\
		\hline\hline
		$m_i \leq \frac{(d-2)_+}{d}$ & Renormalised solution & unknown \\
		\hline
		$\frac{(d-2)_+}{d} < m_i < 2$ & Renormalised solution & Exponential in $L^1$ \\
		\hline
		$2 \leq m_i < \mu - 1 \vphantom{\displaystyle \sum_A^B}$ & 
		{unknown} & \parbox{.35\textwidth}{\smallskip Exponential in $L^1$ for approximate solutions {\it uniformly in the approximation parameter}\smallskip} \\
		\hline
		$\mu - 1 \leq m_i \leq \mu - \frac{4}{d+2}$ & Weak solution & Exponential in $L^1$ \\
		\hline
		$\mu - \frac{4}{d+2} < m_i$ & Bounded solution & Exponential in any $L^p$, $p < \infty$ \\
		\hline
	\end{tabular}

	\begin{tabular}{c}
		\hfill
	\end{tabular}
	\caption{Summarisation of global existence and convergence to equilibrium for mass action reaction--diffusion systems \eqref{mass-action-system}--\eqref{reaction-term} with non-linear diffusion in various parameter regimes.}
	\label{tablesummarisation}
	\end{table}
	
	{The rest of this paper is organised as follows.} In Section \ref{sec:existence}, we show the global existence of renormalised solutions for the general system \eqref{e0}. The proofs of Theorems \ref{thm2}, \ref{thm3}, and \ref{thm4} on the convergence to equilibrium of chemical reaction networks are presented in Section \ref{sec:convergence}. Some technical proofs of auxiliary results are postponed to the Appendix.  

\section{Global existence of renormalised solutions}\label{sec:existence}

	\subsection{Existence of approximate solutions}\label{sec:approx}
	Let $\varepsilon>0$ and consider the following approximating system for $u^{\varepsilon} = (u_1^{\varepsilon}, \ldots, u_I^{\varepsilon})$,
	\begin{equation}\label{approx-system}
			\partial_tu_i^{\varepsilon} - d_i\Delta (u_i^{\varepsilon})^{m_i} = f_i^\varepsilon(u^{\varepsilon})\coleq \frac{f_i(u^{\varepsilon})}{1 + \varepsilon|f(u^{\varepsilon})|}, \quad \nabla (u_i^{\varepsilon})^{m_i}\cdot \nu = 0, \quad u_{i}^{\varepsilon}(\cdot,0) = u_{i,0}^{\varepsilon},
	\end{equation}
	where $u_{i,0}^{\varepsilon}\in L^{\infty}(\Omega)$ is non-negative and $u_{i,0}^{\varepsilon} \xrightarrow{\eps\to0} u_{i,0}$ in $L^1(\Omega)$. Moreover, we demand
	\begin{equation}\label{Ini-Log-Bound}
		\sup_{\varepsilon>0}\int_{\Omega}u_{i,0}^{\varepsilon}\log u_{i,0}^{\varepsilon} \, dx < +\infty.
	\end{equation}
	With this approximation, it is easy to check that the approximated non-linearities $f_i^\eps$ still satisfy the assumptions \eqref{F1}--\eqref{F3}.
	\begin{lemma}\label{existence-approx}
		For any $\varepsilon>0$, there exists a global weak solution $u^{\varepsilon} = (u_i^{\varepsilon})_{i} \in L^\infty_{loc}(0, \infty; L^\infty(\Omega))^I$ to the approximate system \eqref{approx-system} with $\na (\ue_i)^{\frac{m_i}{2}} \in L^2_{loc}(0,\infty;L^2(\Omega))$ and $\partial_tu_i^\varepsilon \in L^2_{loc}(0, \infty; (H^1(\Omega))')$.  
		In detail, 
		\begin{align}
		\begin{aligned} \label{regulweak}
			\int_{t_0}^{t_1} \bigg\langle \frac{d}{dt} u_i^\varepsilon, \psi \bigg\rangle_{(H^1(\Omega))', H^1(\Omega)} dt = &- d_i m_i \int_{t_0}^{t_1} \int_\Omega (u_i^\varepsilon)^{m_i - 1} \nabla u_i^\varepsilon \cdot \nabla \psi \, dx \, dt \\
			&+ \int_{t_0}^{t_1} \int_\Omega \frac{f_i(u^\varepsilon)}{1 + \varepsilon |f(u^\varepsilon)|} \psi \, dx \, dt
		\end{aligned}
		\end{align}
		for all $\psi \in L^2_\mathrm{loc}(0,\infty; H^1(\Omega))$, $i \in \{1,\dotsc,I\}$, and a.e.\ $0 \leq t_0 < t_1$.
		Moreover, each $u_i^\varepsilon$ is a.e.\ non-negative and
		\begin{equation}\label{a0}
			\begin{gathered}
				\sum_{i=1}^{I}\int_{\Omega}u_i^{\varepsilon}(t_1) [\log{u_i^{\varepsilon}(t_1)} + \mu_i - 1] \, dx + C\int_{t_0}^{t_1}\int_{\Omega}\sum_{i=1}^{I}\left|\nabla (u_i^{\varepsilon})^{\frac{m_i}{2}}\right|^2 \, dx \, dt \\
				\leq e^{C(t_1 - t_0)}\left(\int_{\Omega}\sum_{i=1}^{I}u_i^{\varepsilon}(t_0) [\log u_i^{\varepsilon}(t_0) + \mu_i - 1] \, dx + C(t_1 - t_0)\right)
			\end{gathered}
		\end{equation}
		for a.e.\ $0 \leq t_0 < t_1$.
	\end{lemma}
	\begin{proof}
		For each $\varepsilon>0$, we see by recalling $|f(u^{\varepsilon})| = \sum_{i=1}^{I}|f_i(u^{\varepsilon})|$ that
		\begin{equation*}
			\frac{f_i(u^{\varepsilon}(x,t))}{1+\varepsilon|f(u^{\varepsilon}(x,t))|} \leq \frac{1}{\varepsilon} \quad \text{ for all } \quad (x,t)\in Q_T.
		\end{equation*}
		The existence of bounded weak solutions to \eqref{approx-system} is therefore standard. However, since we are not able to find a precise reference, a proof is given in Appendix \ref{appendix}. 
		By multiplying \eqref{approx-system} by $\log(u_i^{\varepsilon}) + \mu_i$ (or more rigorously by $\log(u_i^\eps+\delta) + \mu_i$ for some $\delta>0$, then let $\delta\to0$), summing the resultants over $i=1,\ldots, I$, and then integrating over $\Omega$, we obtain
		\begin{equation}\label{a1}
			\begin{aligned}
			&\frac{d}{dt}\int_{\Omega}\sum_{i=1}^{I}u_i^{\varepsilon}(\log u_i^{\varepsilon} +\mu_i - 1) \, dx + \sum_{i=1}^{I}\frac{4d_i}{m_i}\int_{\Omega}\left|\nabla (u_i^{\varepsilon})^{\frac{m_i}{2}}\right|^2 dx\\
			&\qquad = \frac{1}{1+\varepsilon|f(u^{\varepsilon})|}\int_{\Omega}\sum_{i=1}^{I}f_i(u^{\varepsilon})(\log u_i^{\varepsilon} + \mu_i) \, dx\\
			&\qquad \leq C\sum_{i=1}^{I}\int_{\Omega}u_i^{\varepsilon}(\log u_i^{\varepsilon} + \mu_i - 1) \, dx + C
			\end{aligned}			
		\end{equation}
		where we used \eqref{F3} and $x\leq \delta x\log x + C_\delta$ for all $x\geq 0$ and any $\delta > 0$ at the last step. 
		Hence, by integrating \eqref{a1} over $(t,T)$ and using Gronwall's inequality, we obtain the desired estimate \eqref{a0}. The uniform bound on $u_i^{\varepsilon}\log u_i^{\varepsilon}$ in $L^{\infty}(0,T;L^1(\Omega))$ and $|\nabla (u_i^{\varepsilon})^{\frac{m_i}{2}}|$ in $L^2(0,T;L^2(\Omega))$ follow immediately from \eqref{a0}.	
	\end{proof}
	\subsection{Existence of renormalised solutions}\label{sec:limit}
	As it can be seen from Lemma \ref{existence-approx}, the {\it a priori} estimates of $u_i^{\varepsilon}$ are not enough to extract a convergent subsequence. 
{\color{black} 
	Following the idea from \cite{Fis15}, we consider another approximation of $u_i^{\varepsilon}$ by defining 
	\begin{equation}\label{omega}
	\omega_i^E(v) = (v_i - 3E) \, \zeta\left(\,\sum_{j=1}^{I} \frac{v_j}{E^{\alpha^i_j}} - 1\right) + 3E
	\end{equation}
	for $E \in \mathbb N$, and a smooth function $\zeta: \mathbb R \to [0,1]$ satisfying $\zeta \equiv 1$ on $(-\infty, 0)$ and $\zeta \equiv 0$ on $(1,\infty)$. The constants $\alpha^i_j$ are given by
	\begin{equation}\label{alphas}
		\alpha^i_i \coleq 1, \quad \alpha^i_j \coleq \frac{2}{\min (m_j, 2 - m_j)} \mbox{\ for\ } j \neq i.
	\end{equation}
	\begin{remark}
		The $\alpha^i_j$ are chosen in \eqref{alphas} for the sake of simplicity. In fact, we can choose any $\alpha^i_j$ such that
		\begin{equation*}
			\alpha_i^i = 1, \quad \text{ and } \quad \alpha^i_j \geq \max\left\{\frac{2-m_i}{2-m_j}; \frac{m_i}{m_j}\right\}.
		\end{equation*}
		See the proof of \ref{E2} in Lemma \ref{properties-omega}.
	\end{remark}
	\begin{lemma}\label{properties-omega}
		The smooth truncations $\omega_i^E$ defined in \eqref{omega} have the following properties:
		\begin{enumerate}[label=\emph{(E\theenumi)},ref={(E\theenumi)}]
			\item\label{E1} $\omega_i^E\in C^2(\mathbb R_+^I)$.
			\item\label{E2} For any $i\in \{1,\ldots, I\}$, there exists a constant $K_i>0$ such that the inequality
			\begin{equation}\label{key_estimate}
				\max_{1\leq j,k\leq I}\sup_{E\geq 1}\sup_{v\in \mathbb R_+^I}v_j^\frac{m_j}{2} v_{k}^{1 - \frac{m_k}{2}} |\partial_j \partial_k \omega_i^E(v)| \leq K_i.
			\end{equation}
			\item\label{E3} For every $E$ the set $\supp \ D\omega_i^E$ is bounded.
			\item\label{E4} For all $j\in \{1,\ldots, I\}$ and all $v\in \mathbb R_+^{I}$ there holds $\lim_{E\to\infty}\partial_j\omega_i^E(v)= \delta_{ij}$. 
			\item\label{E5} There exists a constant $K>0$ such that
			\begin{equation*}
				\max_{1\leq j\leq I}\sup_{E\geq 1}\sup_{v\in \mathbb R_+^I}|\partial_j\omega_i^E(v)| \leq K.
			\end{equation*}
			\item\label{E6} $\omega_i^E(v) = v_i$ for any $v\in \mathbb R_+^I$ with $\sum_{j=1}^{I} v_j E^{-\alpha^i_j} \leq 1$.
			\item\label{E7} For every $K>0$ and every $j, k$ there holds 
			\begin{equation*}
				\lim_{E\to \infty}\sup_{|v| \leq K}|\partial_j\partial_k\omega_i^E(v)| = 0.
			\end{equation*}
			\item\label{E8}
			$\omega_i^E(v) = v_i$ for any $v\in \mathbb R_+^I$ with $\sum_{j=1}^{I} \omega_j^E(v) E^{-\alpha^i_j} \leq 1$.
		\end{enumerate}
	\end{lemma}
	\begin{proof}
		Properties \ref{E1}, \ref{E3}, \ref{E4}, and \ref{E6} are immediate. To prove \ref{E2}, we first compute
		\begin{align*}
			\partial_i \omega_i^E(v) &= (v_i - 3E) \zeta'(\cdots) E^{-\alpha^i_i} + \zeta (\cdots),\\
			\partial_j \omega_i^E(v) &= (v_i - 3E) \zeta'(\cdots) E^{-\alpha^i_j}
		\end{align*}
		for $j \neq i$, where for brevity we write $(\cdots)$ instead of $\big(\sum_{j=1}^I E^{-\alpha^i_j} v_j - 1 \big)$. From that, one further gets the second derivatives
		\begin{align*}
			\partial_i^2 \omega_i^E(v) &= (v_i - 3E) \zeta''(\cdots) E^{-2\alpha^i_i} + 2 \zeta' (\cdots) E^{-\alpha^i_i},\\
			\partial_i \partial_j \omega_i^E(v) &= (v_i - 3E) \zeta'' (\cdots) E^{-\alpha_i^i - \alpha^i_j} + \zeta' (\cdots) E^{-\alpha^i_j},\\
			\partial_j \partial_k \omega_i^E(v) &= (v_i - 3E) \zeta'' (\cdots) E^{-\alpha^i_j - \alpha^i_k}
		\end{align*}
		for $j \neq i$ and $k \neq i$.

		To show \eqref{key_estimate}, we will consider the following cases:
		\begin{itemize}
			\item When $i = j = k$, we have
			\begin{align*}
				v_i^{\frac{m_i}{2}}v_i^{1-\frac{m_i}{2}}|\partial_i^2 \varphi_i^E(v)| &= v_i|\partial_i^2\varphi_i^E(v)|\\
				&\leq |v_i||v_i-3E|E^{-2\alpha_i^i}|\zeta''(\cdots)| + 2v_iE^{-\alpha_i^i}|\zeta'(\cdots)|.
			\end{align*}
			Employing the identities $\zeta'(\cdots) = \zeta''(\cdots) = 0$ for $v_i \geq 2 E^{\alpha_i^i}$ as well as the bound $|\zeta'(\cdots)| + |\zeta''(\cdots)| \leq C$, we can further estimate 
			\begin{align*}
				v_i^{\frac{m_i}{2}}v_i^{1-\frac{m_i}{2}}|\partial_i^2 \varphi_i^E(v)| &\leq CE^{-\alpha_i^i}|E^{\alpha_i^i} + E| + C.
			\end{align*}
			Therefore, by choosing $\alpha_i^i = 1$, we have the desired estimate \eqref{key_estimate} for $i=j=k$.
			
			\item When $k=i$ and $j\ne i$, we estimate using $\alpha_i^i = 1$, $\zeta'(\cdots) = \zeta''(\cdots) = 0$ when $v_j \geq 2E^{\alpha^i_j}$ or $v_i \geq 2E$, and $|\zeta'(\cdots)| + |\zeta''(\cdots)| \leq C$,
			\begin{align*}
				\qquad&v_j^{\frac{m_j}{2}}v_i^{1-\frac{m_i}{2}}|\partial_j\partial_i\varphi_i^E(v)|\\
				&\leq v_j^{\frac{m_j}{2}}v_i^{1-\frac{m_i}{2}}|\zeta'(\cdots)|E^{-\alpha^i_j} + v_j^{\frac{m_j}{2}}v_i^{1-\frac{m_i}{2}}|\zeta''(\cdots)||v_i - 3E|E^{-\alpha^i_j - 1}\\
				&\leq CE^{\frac{m_j}{2}\alpha^i_j}E^{1-\frac{m_i}{2}}|\zeta'(\cdots)|E^{-\alpha^i_j} + CE^{\frac{m_j}{2}\alpha^i_j}E^{1-\frac{m_i}{2}}|\zeta''(\cdots)||E + 3E|E^{-\alpha^i_j - 1}\\
				&\leq CE^{\alpha^i_j(\frac{m_j}{2}-1) + (1-\frac{m_i}{2})}.
			\end{align*}
			Thus, \eqref{key_estimate} is proved in the case $k=i\ne j$ if we choose $\alpha^i_j$ such that
			\begin{equation}\label{alphaij1}
			\alpha^i_j(\frac{m_j}{2}-1) + (1-\frac{m_i}{2}) \leq 0 \quad \text{ or equivalently } \quad \alpha^i_j \geq \frac{2-m_i}{2-m_j}.
			\end{equation}
		\item When $j=i$ and $k\ne i$, we estimate similarly to the previous cases
		\begin{align*}
			v_i^{\frac{m_i}{2}}v_k^{1-\frac{m_k}{2}}|\partial_i\partial_k\varphi_i^E(v)| \leq CE^{\frac{m_i}{2} - \alpha^i_k\frac{m_k}{2}}.
		\end{align*}
		If we choose
		\begin{equation}\label{alphaij2}
			\frac{m_i}{2} - \alpha^i_k\frac{m_k}{2} \leq 0 \quad \text{ or equivalently } \quad \alpha^i_k \geq \frac{m_i}{m_k},
		\end{equation}
		then obviously \eqref{key_estimate} holds true for $j=i\ne k$.
		
		\item Finally, when $j\ne i$ and $k\ne i$, we estimate
		\begin{align*}
			v_j^{\frac{m_j}{2}}v_k^{1-\frac{m_k}{2}}|\partial_j\partial_k\varphi_i^E(v)|
			&\leq CE^{\alpha^i_j \frac{m_j}{2}}E^{\alpha^i_k(1-\frac{m_k}{2})}|\zeta''(\cdots)|E^{-\alpha^i_j - \alpha^i_k}E\\
			&\leq CE^{\alpha^i_j(\frac{m_j}{2}-1) - \alpha^i_k \frac{m_k}{2}+1}.
		\end{align*}
		It is easy to see that from \eqref{alphaij1} and \eqref{alphaij2} it follows
		\begin{equation*}
			\alpha^i_j\Big(\frac{m_j}{2}-1\Big) - \alpha^i_k \frac{m_k}{2}+1 \leq 0,
		\end{equation*}
		and hence \eqref{key_estimate} is proved in this case.
	\end{itemize}
				
	\medskip
	From \eqref{alphaij1} and \eqref{alphaij2}, we see that $\alpha^i_j > 1$ for all $j\ne i$, hence, \ref{E5} and \ref{E7} follow.
	
	\medskip
		For \ref{E8}, it is enough to show that $\sum_{j=1}^{I} \omega_j^E(v) E^{-\alpha^i_j} \leq 1$ yields $\sum_{j=1}^{I} v_j E^{-\alpha^i_j} \leq 1$. Set
		\begin{equation*}
			y \coleq \sum_{j=1}^{I} \frac{v_j}{E^{\alpha^i_j}} - 1 \quad \text{ and } \quad z \coleq \sum_{j=1}^{I} \frac{1}{E^{\alpha^i_j}}.
		\end{equation*}
		From $\omega_j^E(v) = (v_j - 3E)\zeta(y) + 3E$ and $\sum_{j=1}^{I} \omega_j^E(v) E^{-\alpha^i_j} \leq 1$ we deduce
		\begin{equation*}
			\zeta(y) (y+1) + 3E (1 - \zeta(y)) z \leq 1.
		\end{equation*}
		Due to $\zeta \leq 1$ and $E z \geq 1$, we are now led to
		\begin{equation*}
			y\zeta(y) \leq (1-3 E z) (1-\zeta(y)) \leq 0.
		\end{equation*}
		This entails $y \leq 0$ and, thus, completes the proof of the lemma.
	\end{proof}
}
	\begin{lemma}\label{a.e.-limit}
		Consider non-negative functions $u_0 = (u_{i,0})_i \in L^1(\Omega)^I$ which satisfy
		\[
		\sum_{i=1}^I \int_\Omega u_{i,0} \log u_{i,0} \, dx < \infty.
		\]
		Let $u^\varepsilon = (u_1^\varepsilon, \dotsc, u_I^\varepsilon)$ for $\varepsilon \rightarrow 0$ be the sequence of solutions to the regularised problems as stated in Lemma \ref{existence-approx}. 
		
		Then, there exists a subsequence $u^\varepsilon$ converging a.e.\ on $\Omega \times [0,\infty)$ to a limit $u \in L^\infty_{loc}(0,\infty; L^1(\Omega))^I$ with $u_i \log u_i \in L^\infty_{loc}(0,\infty; L^1(\Omega))$ for all $i \in \{1, \dotsc, I\}$.
		Furthermore, $(u_i^\varepsilon)^\frac{m_i}{2} \rightharpoonup (u_i)^\frac{m_i}{2}$ weakly in $L^2(0,T; H^1(\Omega))$ for all $i \in \{1, \dotsc, I\}$ and $T>0$.
	\end{lemma}
	\begin{proof}
		Due to the lack of uniform-in-$\eps$ estimates of the non-linearities, it is difficult to show directly, for instance by means of an Aubin--Lions lemma, that $u_i^{\varepsilon}$ has a convergent subsequence. Following the ideas from \cite{Fis15}, we first prove that $\omega^E_i(u^{\varepsilon})$ converges (up to a subsequence) to $z_i^E$ as $\varepsilon \to 0$, and then that $z_i^E$ converges to $u_i$ (up to a subsequence) as $E\to\infty$. In combination with the convergence of $\omega_i^E(u^{\varepsilon})$ to $u^{\varepsilon}_i$ for $E\to \infty$, this leads to the desired result.
		
		\medskip
		Due to the bound $\varphi_i^E \leq 3E$, it follows at once that $\{\omega_i^E(u^\varepsilon)\}$ is bounded in $L^2(0,T; L^2(\Omega))$ uniformly in $\varepsilon>0$ for each $E \in \mathbb N$. 
		Next, by the chain rule we have
		\begin{equation*}
			\nabla \omega_i^{E}(u^{\varepsilon}) = \sum_{j=1}^{I}\partial_j\omega_i^{E}(u^{\varepsilon})\nabla u_j^{\varepsilon} = \sum_{j=1}^{I}\frac{2}{m_j}\partial_j\omega_i^E(u^{\varepsilon})\left(u_j^\varepsilon\right)^{1-\frac{m_j}{2}}\nabla(u_j^{\varepsilon})^{\frac{m_j}{2}}.
		\end{equation*}
		Thus, 
		\begin{equation*}
			\begin{aligned}
			\int_{Q_T}|\nabla \omega_i^E(u^\varepsilon)|^2 \, dx \, dt  &\leq I\sum_{j=1}^{I}\frac{4}{m_j^2}\int_{Q_T}|\partial_j\omega_i^E(u^\varepsilon)|^2|u_j^\varepsilon|^{2-m_j}\left|\nabla(u_j^\varepsilon)^{\frac{m_j}{2}}\right|^2 \, dx \, dt \\
			&\leq C(E)\sum_{j=1}^{I}\int_{Q_T}\left|\nabla(u_j^\varepsilon)^{\frac{m_j}{2}}\right|^2 \, dx \, dt 
			\end{aligned}
		\end{equation*}
		thanks to the compact support of $D\omega_i^E$, $m_j< 2$, and the $L^2(0,T;L^2(\Omega))$ bound on $\nabla (u_j^{\varepsilon})^{\frac{m_j}{2}}$ from Lemma \ref{existence-approx}. As a consequence, we know that $\{\omega_i^E(u^\varepsilon)\}$ is bounded (uniformly in $\varepsilon$) in $L^{2}(0,T;H^{1}(\Omega))$.
		
		To apply the Aubin--Lions Lemma, we need an estimate concerning the time derivative of $\omega_i^{E}(u^{\varepsilon})$. Using $\partial_j \omega_i^{E}(u^{\varepsilon}) \psi$ as a test function in \eqref{regulweak} and summing over $j\in \{1,\ldots, I\}$, we obtain for almost all $t_2 > t_1 \geq 0$, 
		\begin{align}
			&\hspace{-2em} \int_{\Omega}\omega_i^E(u^{\varepsilon}(\cdot, t_2))\psi(\cdot, t_2) \, dx - \int_{\Omega}\omega_i^E(u^{\varepsilon}(\cdot, t_1))\psi(\cdot, t_1) \, dx - \int_{t_1}^{t_2}\int_{\Omega}\omega_i^E(u^{\varepsilon})\partial_t\psi \, dx \, dt\nonumber \\
			= - &\sum_{j,k=1}^{I}\frac{4d_j}{m_k}\int_{t_1}^{t_2}\int_{\Omega}\psi\left[\partial_j\partial_k\omega_i^E(u^{\varepsilon})(u_j^{\varepsilon})^{\frac{m_j}{2}}(u_k^{\varepsilon})^{1-\frac{m_k}{2}}\right]{\nabla(u_j^{\varepsilon})^{\frac{m_j}{2}}}\cdot {\nabla(u_k^{\varepsilon})^{\frac{m_k}{2}}} \, dx \, dt\nonumber \\
			\quad - &\sum_{j=1}^{I} 2 d_j \int_{t_1}^{t_2}\int_{\Omega}\partial_j\omega_i^E(u^{\varepsilon})(u_j^{\varepsilon})^{\frac{m_j}{2}}\nabla (u^{\varepsilon}_j)^{\frac{m_j}{2}} \cdot \nabla \psi \, dx \, dt\nonumber\\
			\quad + & \sum_{j=1}^{I}\int_{t_1}^{t_2}\int_{\Omega}\partial_j\omega_i^E(u^{\varepsilon})\frac{f_j(u^{\varepsilon})}{1+\varepsilon|f(u^{\varepsilon})|}\psi \, dx \, dt.\label{omega-E}
		\end{align}
		The third term on the right hand side is clearly bounded uniformly in $\varepsilon$ for each fixed $E \in \mathbb N$ since $D\omega_i^E$ has a compact support. The first and second terms are bounded uniformly in $\varepsilon$ thanks to the boundedness of $\nabla(u_k^\varepsilon)^{\frac{m_k}{2}}$ in $L^2(0,T;L^2(\Omega))$ for all $k=1,\ldots, I$, and properties \ref{E2}, \ref{E3} in Lemma \ref{properties-omega}. It follows then that $\partial_t\omega_i^{E}(u^{\varepsilon})$ is bounded uniformly (w.r.t.\ $\varepsilon>0$) in $L^1(0,T;(W^{1,\infty}(\Omega))')$ for each fixed $E \in \mathbb N$.
		
		Therefore, by applying an Aubin--Lions lemma to the sequence $\{\omega_i^E(u^\varepsilon)\}_{\varepsilon>0}$, for fixed $E \in \mathbb N$, there exists a subsequence (not relabeled) of $\omega_i^{E}(u^{\varepsilon})$ converging strongly in $L^2(0,T; L^2(\Omega))$ and, thus, almost everywhere as $\varepsilon\to 0$. Using a diagonal sequence argument, one can extract a further subsequence such that $\omega_i^{E}(u^{\varepsilon})$ converges a.e.\ in $\Omega$ to a measurable function $z_i^{E}$ for all $E \in \mathbb N$ and $i = 1, \dotsc, I$.
		
		\medskip
		We next prove that $z_i^E$ converges a.e.\ to some measurable function $u_i$ as $E\to\infty$.		
		First, since $\sum_{j=1}^{I}u_j^{\varepsilon}\log u_j^{\varepsilon}$ is uniformly bounded w.r.t.\ to $\varepsilon>0$ in $L^{\infty}(0,T;L^1(\Omega))$, it follows that $\omega_i^E(u^{\varepsilon})\log \omega_i^{E}(u^{\varepsilon})$ is uniformly bounded w.r.t.\ to $\varepsilon>0$ and $E \in \mathbb N$ in $L^{\infty}(0,T;L^1(\Omega))$. This is trivial in case $E \leq 1$, $\varphi_i^E(u^\varepsilon) \leq 1$, or $\sum_{j=1}^I u_j^\varepsilon E^{-\alpha^i_j} \leq 1$. Otherwise, there exists some $j$ such that $u_j^\varepsilon E^{-\alpha^i_j} > 1/I$ holds true. And as $\alpha^i_j > 1$, we derive
		\[
			\varphi_i^E(u^\varepsilon) \leq 3 E \leq 3 E I u_j^\varepsilon E^{-\alpha^i_j} \leq 3 I u_j^\varepsilon.
		\]
		The bound $\varphi_i^E(u^\varepsilon) > 1$ ensures $\varphi_i^E(u^\varepsilon) \log \varphi_i^E(u^\varepsilon) \leq 3 I u_j^\varepsilon \log ( 3 I u_j^\varepsilon )$ and the claim follows. Thus, $z_i^E\log z_i^E$ is uniformly bounded in $L^{\infty}(0,T;L^1(\Omega))$ w.r.t.\ $E \in \mathbb N$ by Fatou's Lemma.		
		Secondly, $z_j^E(x,t) = z_j^{\widetilde{E}}(x,t)$ for all $\widetilde{E}>E$, if $\sum_{j=1}^{I}z_j^E(x,t)E^{-\alpha_j^i} < 1$. Indeed, 
		\[
			\sum_{j=1}^{I}z_j^E(x,t)E^{-\alpha_j^i} = \lim_{\varepsilon\to0}\sum_{j=1}^{I}\omega_j^{E}(u^{\varepsilon}(x,t))E^{-\alpha_j^i},
		\]
		guarantees that $\sum_{j=1}^{I}\omega_j^{E}(u^{\varepsilon}(x,t))E^{-\alpha_j^i} < 1$ holds true for sufficiently small $\varepsilon>0$. Thanks to the property \ref{E8} in Lemma \ref{properties-omega}, it follows that $\omega_i^{E}(u^{\varepsilon}(x,t)) = u^{\varepsilon}_i(x,t) = \omega_i^{\widetilde{E}}(u^{\varepsilon}(x,t))$ for small enough $\varepsilon$ and all $\widetilde{E} > E$. Therefore, due to $z_i^E(x,t) = \lim_{\varepsilon\to 0}\omega_i^E(u^{\varepsilon}(x,t))$, we obtain $z_i^{E}(x,t) = z_i^{\widetilde{E}}(x,t)$ for all $\widetilde{E}>E$ as desired. As a result, if $\sum_{j=1}^{I}z_j^E(x,t)E^{-\alpha_j^i} < 1$ for some $(x,t)$, then $\lim_{E\to \infty}z_i^E(x,t)$ exists and is finite for all $i=1,\ldots, I$.		
		Using the fact that $\sum_{j=i}^{I}z_j^E\log z_j^E$ is bounded uniformly w.r.t.\ $E \in \mathbb N$ in $L^{\infty}(0,T;L^1(\Omega))$, we have
		\begin{equation*}
			\lim_{E\to\infty}\mathcal L^{n+1}\left(\left\{(x,t)\in Q_T: \sum_{j=1}^Iz_j^E(x,t)E^{-\alpha_j^i} \geq 1 \right\}\right) = 0
		\end{equation*}
		where $\mathcal L^{n+1}$ is the Lebesgue measure in $\mathbb R^{n+1}$. Hence, 
		the limit $u_i(x,t) = \lim_{E\to\infty}z_i^E(x,t)$ exists for a.e.\ $(x,t)\in Q_T$. Moreover, $u_i\log u_i \in L^{\infty}(0,T;L^1(\Omega))$ due to Fatou's Lemma and $z_i^E\log z_i^E\in L^{\infty}(0,T;L^1(\Omega))$.
		Since $\sum_{j=1}^{I}u_j^{\varepsilon}$ is uniformly bounded in $L^{1}(Q_T)$, we find
		\[
			\begin{aligned}
			0 &= \lim_{E\to\infty}\mathcal L^{n+1}\left(\left\{(x,t)\in Q_T:\sum_{j=1}^{I}u_j^{\varepsilon}(x,t)E^{-\alpha_j^i}\geq 1\right\}\right)\\
			&\geq \lim_{E\to\infty}\mathcal L^{n+1}\left(\left\{(x,t)\in Q_T: u_i^{\varepsilon}(x,t) \ne \omega_i^{E}(u^{\varepsilon})(x,t) \right\}\right)
			\end{aligned}
		\]
		and this limit is uniform in $\varepsilon>0$. Now, we estimate for $\delta>0$
		\begin{align*}
			&\mathcal L^{d+1} \bigg( \bigg\{ (x,t) \in Q_T : |u_i^\varepsilon(x,t) - u_i(x,t)| > \delta \bigg\} \bigg) \\
			&\quad \leq \mathcal L^{d+1} \bigg( \bigg\{ (x,t) \in Q_T : u_i^\varepsilon(x,t) \neq \varphi_i^E(u^\varepsilon)(x,t) \bigg\} \bigg) \\
			&\qquad + \mathcal L^{d+1} \bigg( \bigg\{ (x,t) \in Q_T : |\varphi_i^E(u^\varepsilon)(x,t) - z_i^E(x,t)| > \frac{\delta}{2} \bigg\} \bigg) \\
			&\qquad + \mathcal L^{d+1} \bigg( \bigg\{ (x,t) \in Q_T : |z_i^E(x,t) - u_i(x,t)| > \frac{\delta}{2} \bigg\} \bigg).
		\end{align*}
		As we have proved that $\omega_i^E(u^{\varepsilon}) \to z_i^E$ a.e.\ in $Q_T$ for $\varepsilon\to0$ and fixed $E \in \mathbb N$, and that $z_i^{E} \to u_i$ a.e.\ in $Q_T$ as $E\to\infty$, we infer the convergence in measure of $u_i^\varepsilon$ to $u_i$ for $\eps\to 0$, and convergence a.e.\ of another subsequence.
		
		The uniform bound on $u_i^\varepsilon \log u_i^\varepsilon$ in $L^\infty(0, T; L^1(\Omega))$ and the convergence $u_i^{\varepsilon} \to u_i$ a.e.\ ensure that $u_i^\varepsilon \rightarrow u_i$ strongly in $L^p(0,T; L^1(\Omega))$ for all $p \geq 1$. One can easily prove this by truncating $u_i^\varepsilon$ at a sufficiently large threshold. By the strong convergence of $u_i^\varepsilon$ to $u_i$ in $L^1(0,T; L^1(\Omega))$, we are able to prove the distributional convergence of $(u_i^\varepsilon)^\frac{m_i}{2}$ to $(u_i)^\frac{m_i}{2}$ by noting that $m_i<2$. The uniform $L^2(0,T; H^1(\Omega))$ bound on $(u_i^\varepsilon)^\frac{m_i}{2}$ then leads to a subsequence $(u_i^\varepsilon)^\frac{m_i}{2}$ which converges to $(u_i)^\frac{m_i}{2}$ weakly in $L^2(0,T; H^1(\Omega))$. Thanks to the weak lower semicontinuity of the $L^2(0,T; L^2(\Omega))$ norm, we also deduce 
		\[
		\int_0^T \int_\Omega |\nabla (u_i)^\frac{m_i}{2}|^2 \, dx \, dt \leq \liminf_{\varepsilon \rightarrow 0} \int_0^T \int_\Omega |\nabla (u_i^\varepsilon)^\frac{m_i}{2}|^2 \, dx \, dt
		\]
		for all $i \in \{1,\dotsc,I\}$, where the right hand side is bounded via the uniform $L^2(0,T; H^1(\Omega))$ bound on $(u_i^\varepsilon)^\frac{m_i}{2}$. This completes the proof of the lemma. 
	\end{proof}
	
	At this point, we expect the function $u = (u_1, \ldots, u_I)$ in Lemma \ref{a.e.-limit} to be a renormalised solution to \eqref{e0}. We will first derive an equation admitting $\omega_i^E(u)$ as a solution. This equation is already ``almost'' matching the formulation for a renormalised solution \eqref{renormalised-form} except for a ``defect measure''.
	\begin{lemma}
		Let $u = (u_1, \ldots, u_I)$ be the functions constructed in Lemma \ref{a.e.-limit}. Then, for any $\psi \in C^{\infty}([0,T];C_0^{\infty}(\overline \Omega))$, $\omega_i^E(u)$ satisfies
		\begin{equation}\label{defect-measure}
		\begin{aligned}
			&\hspace{-2em}\int_{\Omega}\omega_i^E(u(\cdot, T))\psi(\cdot, T) \, dx - \int_{\Omega}\omega_i^E(u(\cdot, 0))\psi(\cdot, 0) \, dx - \int_{Q_T}\omega_i^E(u)\partial_t\psi \, dx \, dt\\
			= &-\int_{Q_T}\psi\, d\mu_i^{E}(x,t)- \sum_{j=1}^{I}d_j m_j\int_{Q_T}\partial_j\omega_i^E(u)u_j^{m_j-1}\nabla u_j \cdot \nabla \psi \, dx \, dt\\ 
			&+\sum_{j=1}^{I}\int_{Q_T}\partial_j\omega_i^E(u){f_j(u)}\psi \, dx \, dt
		\end{aligned}
		\end{equation}	
		where $\mu_i^E$ is a signed Radon measure satisfying
		\begin{equation}\label{vanish-measure}
			\lim_{E\to\infty}|\mu_i^E|(\overline{\Omega}\times [0,T)) = 0
		\end{equation}
		for all $T>0$ and $i \in \{1, \ldots, I\}$.
	\end{lemma}	
	\begin{proof}
		Choosing $t_1 = 0$, $t_2 = T$, and using the convergence $u^{\varepsilon} \to u$ a.e.\ in $Q_T$, the weak convergence $\nabla (u_i^\varepsilon)^\frac{m_i}{2} \rightharpoonup \nabla (u_i)^\frac{m_i}{2}$ in $L^2(0, T; L^2(\Omega))$, and the fact that $D \omega_i^E(u^{\varepsilon})$ vanishes when $u^{\varepsilon}$ is too large, we straightforwardly obtain the convergence of the left hand side and the last two terms on the right hand side of \eqref{omega-E}. It remains to establish 
		\begin{multline*}
			-4\sum_{j,k=1}^{I}\frac{d_j}{m_k} \int_{Q_T}\psi\left[\partial_j\partial_k\omega_i^E(u^{\varepsilon})({u_j^{\varepsilon}})^{\frac{m_j}{2}}({u_k^{\varepsilon}})^{1-\frac{m_k}{2}}\right] {\nabla({u_j^{\varepsilon}})^{\frac{m_j}{2}}}\cdot {\nabla({u_k^{\varepsilon}})^{\frac{m_k}{2}}} \, dx \, dt \\
			\longrightarrow -\int_{Q_T}\psi\,d\mu_i^E(x,t)
		\end{multline*}
		with a signed Radon measure $\mu_i^E$ satisfying \eqref{vanish-measure}. By denoting
		\begin{equation*}
			\mu_{i,\varepsilon}^E\coleq 4\sum_{j,k=1}^{I}\frac{d_j}{m_k} \left[\partial_j\partial_k\omega_i^E(u^{\varepsilon})({u_j^{\varepsilon}})^{\frac{m_j}{2}}({u_k^{\varepsilon}})^{1-\frac{m_k}{2}}\right] {\nabla({u_j^{\varepsilon}})^{\frac{m_j}{2}}}\cdot {\nabla({u_k^{\varepsilon}})^{\frac{m_k}{2}}} \, dx \, dt,
		\end{equation*}
		we can use property \ref{E2} of the truncation function $\omega_i^E$ and the $\varepsilon$-uniform bound of $\nabla (u_i^{\varepsilon})^{\frac{m_i}{2}}$ in $L^2(0,T;L^2(\Omega))$, to obtain 
		\begin{equation*}
			|\mu_{i,\varepsilon}^{E}|(\overline{Q}_T) \leq C
		\end{equation*}
		for all $\varepsilon>0$. Therefore, by passing to a subsequence, we know that $\mu_{i,\varepsilon}^E$ weak-$\ast$ converges on $\overline{Q}_T$ to a signed Radon measure $\mu_i^{E}$ as $\varepsilon\to 0$. It remains to prove \eqref{vanish-measure}. Due to Young's inequality, we have
		\begin{equation*}
			\begin{aligned}
			&|\mu_{i,\varepsilon}^E|(\overline{Q}_T)\\
			&\leq C\sum_{j,k=1}^{I}\sum_{K=1}^{\infty}\int_{Q_T}|\partial_j\partial_k\omega_i^E(u^{\varepsilon})| ({u_j^{\varepsilon}})^{\frac{m_j}{2}}({u_k^{\varepsilon}})^{1-\frac{m_k}{2}} \chi_{\{K-1\le |u^{\varepsilon}| < K\}}\left|\nabla (u_j^\varepsilon)^{\frac{m_j}{2}}\right|^2 \, dx \, dt\\
			&\leq C\sum_{l=1}^{I}\sum_{K=1}^{\infty}\nu_{l,K}^\varepsilon(\overline{Q}_T) \sup_{\genfrac{}{}{0pt}{1}{1 \leq j,k \leq I}{K-1\leq |v| < K}} |\partial_j\partial_k\omega_i^E(v) |v_j^{\frac{m_j}{2}}v_k^{1-\frac{m_k}{2}}
			\end{aligned}
		\end{equation*}
		where 
		\[
			\nu_{l,K}^{\varepsilon} \coleq \chi_{\{K-1\le |u^{\varepsilon}| < K\}}\left|\nabla (u_l^\varepsilon)^{\frac{m_l}{2}}\right|^2 \, dx \, dt.
		\]
		We stress that $\nu_{l,K}^{\varepsilon}$ is uniformly bounded w.r.t.\ $l$, $K$, and $\varepsilon$. Consequently, we may pass to a subsequence $\nu_{l,K}^{\varepsilon}$ which converges weak-$\ast$ on $\ol Q_T$ to a Radon measure $\nu_{l,K}$ as $\varepsilon \rightarrow 0$. Together with the weak-$\ast$ convergence of $\mu_{i,\varepsilon}^E$ to $\mu_i^E$ on $\overline{Q}_T$, we derive
		\begin{equation*}
			\begin{aligned}
				|\mu_i^E|(\overline{Q}_T) &\leq \liminf\limits_{\varepsilon\to 0}|\mu_{i,\varepsilon}^E|(\overline{Q}_T)\\
				&\leq C\sum_{l=1}^{I}\sum_{K=1}^{\infty} \nu_{l,K}(\overline{Q}_T) \sup_{\genfrac{}{}{0pt}{1}{1 \leq j,k \leq I}{K-1\leq |v| < K}} |\partial_j\partial_k\omega_i^E(v) |v_j^{\frac{m_j}{2}}v_k^{1-\frac{m_k}{2}}.
			\end{aligned}
		\end{equation*}
		Moreover,
		\[
			\sum_{K=1}^{\infty}\nu_{l,K}^{\varepsilon}(\overline{Q}_T) = \int_{Q_T}\left|\nabla({u_j^\varepsilon})^{\frac{m_j}{2}}\right|^2 \, dx \, dt,
		\]
		which is bounded uniformly in $\varepsilon$. Fatou's Lemma applied to the counting measure on $\mathbb N$ now entails
		\begin{equation*}
			\sum_{K=1}^{\infty} \nu_{l,K}(\overline{Q}_T) \leq \liminf_{\varepsilon \rightarrow 0} \sum_{K=1}^{\infty}\nu_{l,K}^{\varepsilon}(\overline{Q}_T) < +\infty.
		\end{equation*}
		Therefore, employing the dominated convergence theorem, we can finally estimate
		\begin{equation*}
			\begin{aligned}
				\lim\limits_{E\to\infty}|\mu_i^E|(\overline{Q}_T) &\leq C\sum_{l=1}^{I}\sum_{K=1}^{\infty} \nu_{l,K} (\overline{Q}_T) \lim\limits_{E\to\infty} \sup_{\genfrac{}{}{0pt}{1}{1 \leq j,k \leq I}{K-1\leq |v| < K}} |\partial_j\partial_k\omega_i^E(v) |v_j^{\frac{m_j}{2}}v_k^{1-\frac{m_k}{2}}\\
				&= C\sum_{l=1}^{I}\sum_{K=1}^{\infty} \nu_{l,K} (\overline{Q}_T) \cdot 0 = 0
			\end{aligned}
		\end{equation*}
		thanks to the properties \ref{E2} and \ref{E7} of the truncation function $\omega_i^E$.
	\end{proof}
	To prove Theorem \ref{existence-renormalised}, we use the following technical lemma whose proof can be found in \cite{Fis15}.
	\begin{lemma}[A weak chain rule for the time derivative]\label{weak-chain}\cite[Lemma 4]{Fis15}
		Let $\Omega$ be a bounded domain with Lipschitz boundary. Assume that $T>0$, $v\in L^2(0,T;H^1(\Omega)^I)$, and $v_0\in L^1(\Omega)^I$. Let $\nu_i$ be a Radon measure on $\overline{Q}_T$, $w_i \in L^1(Q_T)$, and $z_i\in L^2(0,T;L^2(\Omega)^I)$ for $1\leq i\leq I$. Assume moreover that for any $\psi\in C^{\infty}(\overline{Q}_T)$ with compact support we have 
		\begin{multline*}
			\int_{\Omega}v_i(T)\psi(T) \, dx -\int_{Q_T}v_i\frac{d}{dt}\psi \, dx \, dt - \int_{\Omega}v_0\psi(0) \, dx \\ = \int_{\overline{Q}_T}\psi \, d\nu_i + \int_{Q_T}w_i\psi  \, dx \, dt + \int_{Q_T}z_i\cdot \nabla\psi  \, dx \, dt.
		\end{multline*}
		Let $\xi:\mathbb R^I \to \mathbb R$ be a smooth function with compactly supported first derivatives. Then, we have for all $\psi\in C^{\infty}(\overline{Q}_T)$ with compact support
		\begin{equation*}
			\begin{aligned}
				\biggl|\int_{\Omega}\xi(v(T))\psi(T) \, dx &-\int_{Q_T}\xi(v)\frac{d}{dt}\psi \, dx \, dt - \int_{\Omega}\xi(v_0)\psi(0) \, dx - \sum_{i=1}^{I}\int_{Q_T}\psi \partial_i\xi(v)w_i \, dx \, dt\\
				&- \sum_{i=1}^{I}\int_{Q_T}\partial_i\xi(v)z_i\cdot \nabla \psi \, dx \, dt - \sum_{i,k=1}^{I}\int_{Q_T}\psi \partial_i\partial_k\xi(v)z_i\cdot \nabla v_k \, dx \, dt\biggr|\\
				&\leq C(\Omega)\|\psi\|_{\infty}(\sup_{u}|D\xi(u)|)\sum_{i=1}^{I}|\nu_i|(\overline{Q}_T).
			\end{aligned}
		\end{equation*}
	\end{lemma}
	We are now ready to prove the existence of global renormalised solutions.
	\begin{proof}[Proof of Theorem \ref{existence-renormalised}]
		By applying Lemma \ref{weak-chain} to \eqref{defect-measure}, we obtain an approximate relation for the weak time derivative of $\xi(\varphi^E(u))$, which leads to the desired equation for $\xi(u)$ when passing to the limit $E \rightarrow \infty$. In detail, we set $v_i \coleq \omega_i^E(u)$, $(v_0)_i \coleq \varphi_i^E(u_0)$, $\nu_i \coleq -\mu_i^E$,
		\begin{equation*}
		z_i\coleq -\sum_{j=1}^{I} d_j m_j \partial_j \omega_i^E(u) u_j^{m_j - 1} \nabla u_j, \quad w_i\coleq \sum_{j=1}^{I} \partial_j \omega_i^E(u) f_j(u),
		\end{equation*}
		and obtain for any smooth function $\xi$ with compactly supported first derivatives
		\begin{equation}\label{final-limit}
			\begin{aligned}
				&\biggl|\int_{\Omega}\xi(\omega^E(u))\psi(T) \, dx -\int_{Q_T}\xi(\omega^E(u))\frac{d}{dt}\psi \, dx \, dt - \int_{\Omega}\xi(\omega^E(u_0))\psi(0) \, dx \\
				&- \sum_{i=1}^{I}\sum_{j=1}^{I}\int_{Q_T}\psi \partial_i\xi(\omega^E(u))\partial_j\omega_i^E(u)f_j(u) \, dx \, dt\\
				&+ \sum_{i=1}^{I}\sum_{j=1}^{I}\int_{Q_T} d_j m_j \partial_i\xi(\omega^E(u)) \partial_j\omega_i^E(u) u_j^{m_j - 1}\nabla u_j\cdot \nabla \psi \, dx \, dt\\
				&+ \sum_{i,k=1}^{I}\sum_{j,\ell=1}^{I}\int_{Q_T} d_j m_j \psi \partial_i\partial_k\xi(\omega^E(u))\partial_j\omega_i^E(u) \partial_{\ell}\omega_k^E(u) u_j^{m_j - 1}\nabla u_j\cdot \nabla u_\ell \, dx \, dt\biggr|\\
				&\leq C(\Omega)\|\psi\|_{\infty}(\sup_{u}|D\xi(u)|)\sum_{i=1}^{I}|\mu_i^E|(\overline{Q}_T).
			\end{aligned}
		\end{equation}		
		We now want to pass to the limit $E\to\infty$ in \eqref{final-limit} to obtain \eqref{renormalised-form}. Note first that due to \eqref{vanish-measure}, we obtain in the limit an equality instead of just an estimate. Since $\xi$ has compactly supported first derivatives, $\partial_j\omega_i^E(u)$ is bounded (see \ref{E5}), and $\omega_i^E(u) \to u_i$ a.e.\ as $E\to\infty$, we directly obtain the following limits for the first line in \eqref{final-limit}:
		\begin{align*}
			\int_{\Omega}\xi(\omega^E(u(T)))\psi(T) \, dx &\longrightarrow \int_{\Omega}\xi(u(T))\psi(T) \, dx,\\
			\int_{Q_T}\xi(\omega^E(u))\frac{d}{dt}\psi \, dx \, dt &\longrightarrow \int_{Q_T}\xi(u)\frac{d}{dt}\psi \, dx \, dt,\\
			\int_{\Omega}\xi(\omega^E(u_0))\psi(0) \, dx &\longrightarrow  \int_{\Omega}\xi(u_0)\psi(0) \, dx.
		\end{align*}
		By employing $m_j u_j^{m_j-1} \nabla u_j = 2 (u_j)^\frac{m_j}{2} \nabla (u_j)^\frac{m_j}{2}$ and $(u_j)^\frac{m_i}{2} \in L^2(0, T; H^1(\Omega))$, we also derive 
		\begin{multline*}
		\sum_{i=1}^{I}\sum_{j=1}^{I}\int_{Q_T} d_j m_j \partial_i\xi(\omega^E(u)) \partial_j\omega_i^E(u) u_j^{m_j-1} \nabla u_j \cdot \nabla \psi \, dx \, dt \\
		\longrightarrow\sum_{i=1}^{I}\int_{Q_T} d_i m_i \partial_i\xi(u) u_i^{m_i-1} \nabla u_i \cdot \nabla \psi \, dx \, dt.
		\end{multline*}
		It remains to ensure the convergence of the third and fourth line of \eqref{final-limit}. To this end, we recall $(u_j)^\frac{m_j}{2} \in L^2(0, T; H^1(\Omega))$, and we utilise the following observation: there exists a constant $E_0>0$ such that for all $E > E_0$ the inequality $\sum_{i=1}^I u_i \geq E_0$ implies $\partial_i \xi(\varphi^E(u)) = \partial_i \xi(u) = 0$ and $\partial_i \partial_k \xi(\varphi^E(u)) = \partial_i \partial_k \xi(u) = 0$ for all $i,k \in \{1,\dotsc,I\}$. The convergence of the third and fourth line above is now a consequence of this auxiliary result, as the derivatives of $\xi$ are zero provided $\max_i u_i$ is larger than $E_0$:
		\begin{align*}
		&\sum_{i=1}^{I}\sum_{j=1}^{I}\int_{Q_T}\psi \partial_i\xi(\omega^E(u))\partial_j\omega_i^E(u)f_j(u) \, dx \, dt \longrightarrow   \sum_{i=1}^{I}\int_{Q_T}\psi \partial_i\xi(u)f_i(u) \, dx \, dt, \\
		&\sum_{i,j,k,\ell=1}^I \int_{Q_T} d_j m_j \psi \partial_i\partial_k\xi(\omega^E(u)) \partial_j\omega_i^E(u) \partial_{\ell}\omega_k^E(u) u_j^{m_j - 1}\nabla u_j \cdot \nabla u_\ell \, dx \, dt \\
		&\qquad\qquad \longrightarrow \sum_{i,k=1}^I \int_{Q_T} d_i m_i \psi \partial_i\partial_k \xi(u) u_i^{m_i - 1} \nabla u_i \cdot \nabla u_k \, dx \, dt.
		\end{align*}
		The previous observation readily follows by choosing $E_0>0$ such that $\supp (D\xi) \subset B_{E_0/\sqrt{I}}(0)$. Given $E > E_0$ and $\sum_{i=1}^I u_i \geq E_0$, one gets $\sum_{i=1}^{I} \varphi_i^E(u) \geq E_0$ from the definition of $\varphi_i^E$ in \eqref{omega}. In particular, $u, \varphi^E(u) \notin B_{E_0/\sqrt{I}}(0)$, which proves the claim.
	\end{proof}
\subsection{Entropy--entropy dissipation law and conservation laws}
We prove in this subsection that any global renormalised solution to \eqref{e0} satisfies the entropy law \eqref{entropy_law} and \eqref{conservation_law}. These relations become helpful in the next section to investigate the convergence to equilibrium of renormalised solutions.

\begin{proof}[Proof of the entropy law \eqref{entropy_law}]
	We will use ideas from \cite[Proposition 5]{Fis17}. Let $M \geq 2$ and $\theta_M: \mathbb R \to \mathbb R $ be a smooth function such that for some $C>0$ sufficiently large,  
	\begin{equation*}
		\theta_M(s) = s \quad \forall s \leq M, \qquad 0\leq \theta_M'(s)\leq 1\quad \forall s\in \mathbb R, \qquad \theta_M'(s) = 0\quad \forall s\geq M^C,
	\end{equation*}
	\begin{equation*}
		\text{and} \quad |\theta_M''(s)| \leq \frac{C}{1+s\log(s+1)} \quad \forall s\geq 0.
	\end{equation*}
	Such a function $\theta_M$ exists for $C>0$ large enough as in this situation 
	\[
	\int_M^{M^C}\frac{1}{1+s\log(s+1)} \, ds \geq 2.
	\]
	
	In the following, we first choose the renormalising function $\xi$ from Definition \ref{def-renormalised} to be 
	\begin{equation*}
		\xi(u)\coleq \theta_M\left(\sum_{i=1}^I (u_i+\varepsilon)(\log(u_i+\varepsilon) + \mu_i - 1)\right),
	\end{equation*}
	then let $\varepsilon\to 0$ and $M\to \infty$ consecutively.	
	Using this $\xi$ in \eqref{renormalised-form} with $\psi \equiv 1$, we derive
	\begin{align}
		&\int_{\Omega}\theta_M\left(\sum_{i=1}^I (u_i(\tau)+\varepsilon)(\log(u_i(\tau)+\varepsilon) + \mu_i - 1)\right) dx\bigg|_{\tau=s}^{\tau=T} \nonumber \\
		&= -\sum_{i=1}^Id_im_i\int_s^T\int_{\Omega}\theta_M'(\cdots)\frac{1}{u_i+\varepsilon}u_i^{m_i-1}|\nabla u_i|^2 \, dx \, dt \nonumber \\
		&- \sum_{i,j=1}^Id_jm_i\int_s^T\int_{\Omega}\theta_M''(\cdots)(\log(u_i+\varepsilon)+\mu_i)(\log(u_j+\varepsilon)+\mu_j)u_i^{m_i-1}\nabla u_i\nabla u_j \, dx \, dt \nonumber \\
		&+\sum_{i=1}^I\int_s^T\int_{\Omega}\theta_M'(\cdots)(\log(u_i+\varepsilon)+\mu_i)f_i(u) \, dx \, dt. \label{b1}
	\end{align}	
	Here we write $(\cdots)\coleq \sum_{i=1}^I(u_i+\varepsilon)(\log(u_i+\varepsilon) + \mu_i - 1)$ for simplicity. In the limit $\varepsilon\to 0$, the convergence of the left hand side of \eqref{b1} is immediate by the Lebesgue dominated convergence theorem. By rewriting
	\begin{equation*}
		\frac{u_i^{m_i-1}}{u_i+\varepsilon}|\nabla u_i|^2 = \frac{4}{m_i^2}\frac{u_i}{u_i+\varepsilon}\left|\nabla u_i^{\frac{m_i}{2}}\right|^2,
	\end{equation*}
	and
	\begin{equation*}
		u_i^{m_i-1}\nabla u_i\nabla u_j = \frac{4}{m_im_j} u_i^{\frac{m_i}{2}}u_j^{1-\frac{m_j}{2}}\nabla u_i^{\frac{m_i}{2}} \nabla u_j^{\frac{m_j}{2}},
	\end{equation*}
	we can use the fact that $\nabla u_i^{\frac{m_i}{2}}$ is bounded in $L^2(Q_T)$ and $\theta_M'(s) = \theta_M''(s) = 0$ for $s\geq M^C$ to pass to the limit $\varepsilon\to 0$ for the first and second terms on the right hand side of \eqref{b1}. Concerning the last term, we observe that
	\begin{equation*}
		\theta_M'(\cdots) f_i(u)\log (u_i+\varepsilon) \leq C(M)u_i|\log (u_i+\varepsilon)|
	\end{equation*} 
	thanks to the Lipschitz continuity of $f_i$ and the fact that $\theta_M'(s) = 0$ for $s\geq M^C$. Note that $\{u_i|\log (u_i+\varepsilon)|\}_{\eps>0}$ is bounded uniformly (in $\varepsilon$) in $L^1(Q_T)$ due to Lemma \ref{a.e.-limit}. We now use the assumption \eqref{entropy-inequality} and Fatou's Lemma to get
	\begin{align*}
		\limsup_{\varepsilon\to 0}\sum_{i=1}^I\int_s^T\int_{\Omega}\theta_M'(\cdots)(\log(u_i+\varepsilon)+\mu_i)f_i(u) \, dx \, dt\\
		\leq \sum_{i=1}^I\int_s^T\int_{\Omega}\theta_M'(\cdots)(\log u_i + \mu_i)f_i(u) \, dx \, dt.
	\end{align*}
	Therefore, by letting $\varepsilon\to 0$ in \eqref{b1}, we obtain
	\begin{equation}\label{b2}
	\begin{aligned}
		&\int_{\Omega}\theta_M\left(\sum_{i=1}^I u_i(\tau)(\log u_i(\tau) + \mu_i - 1)\right) dx\bigg|_{\tau=s}^{\tau=T}\\
		&\leq -\sum_{i=1}^Id_im_i\int_s^T\int_{\Omega}\theta_M'(\cdots)u_i^{m_i-2}|\nabla u_i|^2 \, dx \, dt\\
		&- \sum_{i,j=1}^Id_jm_i\int_s^T\int_{\Omega}\theta_M''(\cdots)(\log u_i+\mu_i)(\log u_j+\mu_j)u_i^{m_i-1}\nabla u_i\nabla u_j \, dx \, dt\\
		&+\sum_{i=1}^I\int_s^T\int_{\Omega}\theta_M'(\cdots)(\log u_i+\mu_i)f_i(u) \, dx \, dt.
	\end{aligned}	
	\end{equation}
	Note that here $(\cdots)$ stands for $\sum_{i=1}^I u_i(\log u_i+ \mu_i - 1)$. We now pass to the limit $M\to \infty$. The convergence of the left hand side of \eqref{b2} follows from Fatou's Lemma and the dominated convergence theorem thanks to the fact that $u_i\log u_i \in L^\infty(0,T;L^1(\Omega))$. Using $\nabla (u_i)^{\frac{m_i}{2}}\in L^2(Q_T)$ and $u_i^{m_i-2}|\nabla u_i|^2 = \frac{4}{m_i^2} \big|\nabla (u_i)^{\frac{m_i}{2}} \big|^2$, the convergence of the first term on the right hand side of \eqref{b2} follows via dominated convergence. The reaction term can also be dealt by Fatou's Lemma, the entropy dissipation assumption \eqref{entropy-inequality} and the fact that $\theta_M' \geq 0$. We are left to consider the second term on the right hand side of \eqref{b2}. It is straightforward that the integrand converges pointwise to zero as $M\to \infty$, since $\theta_M$ is an approximation of the identity. By the fact that $\nabla (u_i)^{\frac{m_i}{2}}\in L^2(Q_T)$ and rewriting
	\begin{multline*}
		\theta_M''(\cdots)(\log u_i+\mu_i)(\log u_j+\mu_j)u_i^{m_i-1}\nabla u_i\nabla u_j\\
		= \frac{4}{m_im_j}\theta_M''(\cdots)(\log u_i+\mu_i)(\log u_j+\mu_j)u_i^{\frac{m_i}{2}}u_j^{1-\frac{m_j}{2}}\nabla u_i^{\frac{m_i}{2}}\nabla u_j^{\frac{m_j}{2}},
	\end{multline*}
	it is sufficient to show that $\theta_M''(\cdots)(\log u_i+\mu_i)(\log u_j+\mu_j)u_i^{\frac{m_i}{2}}u_j^{1-\frac{m_j}{2}}$ is bounded uniformly in $M$. Indeed, using the assumption on $\theta_M$ as well as $0 < \frac{m_i}{2}$ and $1-\frac{m_j}{2} < 1$, we deduce 
	\begin{multline*}
		\left|\theta_M''(\cdots)(\log u_i+\mu_i)(\log u_j+\mu_j)u_i^{\frac{m_i}{2}}u_j^{1-\frac{m_j}{2}}\right| \\ 
		\leq C \left( 1+\sum\limits_{i=1}^Iu_i(\log(u_i+1))^2 \right)^{-1} \left|(\log u_i+\mu_i)(\log u_j+\mu_j)u_i^{\frac{m_i}{2}}u_j^{1-\frac{m_j}{2}}\right| \leq C,
	\end{multline*}
	In conclusion, by letting $M\to \infty$ in \eqref{b2}, we obtain the desired inequality \eqref{entropy_law}.
\end{proof}

\begin{proof}[Proof of the conservation law \eqref{conservation_law}]
Let $\varrho \in \mathbb R$ be arbitrary. Using the renormalisation 
\begin{equation*}
\xi(u)\coleq \theta_M\left(\varrho \sum_{i=1}^Iq_iu_i + \sum_{i=1}^I(u_i+\varepsilon)(\log(u_i+\varepsilon)+\mu_i-1) \right)
\end{equation*}
and $\psi \equiv 1$ 
in \eqref{renormalised-form}, we can proceed similarly to the previous proof of the entropy law \eqref{entropy_law}, now with the additional information that $\sum_{i=1}^{I}\varrho \theta_M'(\cdots) q_if_i(u) = 0$, to obtain
\begin{align*}
	&\sum_{i=1}^I\int_{\Omega}\varrho q_i u_i(\tau) + u_i(\tau)(\log u_i(\tau) + \mu_i - 1) \, dx \bigg|_{\tau = 0}^{\tau = T}\\
	&\qquad \leq -\frac{4}{m_i^2}\sum_{i=1}^I\int_0^T\int_\Omega d_i|\nabla (u_i)^{m_i/2}|^2 \, dx \, dt + \sum_{i=1}^I\int_0^T\int_{\Omega}f_i(u)(\log u_i + \mu_i) \, dx \, dt.
\end{align*}
For $\varrho > 0$, we divide both sides by $\varrho$ and let $\varrho \to +\infty$ to get
\begin{equation*}
	\sum_{i=1}^I\int_{\Omega}q_iu_i(\tau) \, dx\bigg|_{\tau=0}^{\tau = T} \leq 0.
\end{equation*}
Repeating the argument for $\varrho < 0$, we obtain the reverse inequality and, therefore, the desired conservation law \eqref{conservation_law}.
\end{proof}

\section{Exponential convergence to equilibrium}\label{sec:convergence}
	\subsection{Entropy--entropy dissipation inequality}
	Without loss of generality, we assume that the domain $\Omega$ has unit volume, i.e. $|\Omega| = 1$. We employ the following notation: 
	\begin{equation*}
		\overline{u_i} \coleq \int_{\Omega}u_i \, dx \quad \text{ and } \quad \overline{u} \coleq (\overline{u_i})_{i=1,\ldots, I}.
	\end{equation*}
	To show the convergence to equilibrium for \eqref{mass-action-system}, we exploit the so-called entropy method. 
%
	Assuming the complex balanced condition, system \eqref{mass-action-system} possesses the relative entropy functional 
	\begin{equation}\label{relative-entropy}
	\mathcal{E}[u|u_{\infty}] = \sum_{i=1}^{I}\int_{\Omega}\left(u_i\log\frac{u_i}{u_{i,\infty}}-u_{i}+u_{i,\infty}\right) dx
	\end{equation}
	and the corresponding entropy dissipation function 
	\begin{equation}\label{entropy_dissipation}
	\mathcal{D}[u] = \sum_{i=1}^{I}d_im_i\int_{\Omega}u_i^{m_i - 2}|\nabla u_i|^2 \, dx + \sum_{r=1}^{R}k_ru_{\infty}^{y_r}\int_{\Omega}\Psi\left(\frac{u^{y_r}}{u_{\infty}^{y_r}};\frac{u^{y_r'}}{u_{\infty}^{y_r'}}\right) dx
	\end{equation}
	with $\Psi(x;y) = x\log(x/y) - x + y$ (see \cite[Proposition 2.1]{DFT17} for a derivation of $\mathcal D[u]$). \textit{Formally} we have
	\begin{equation*}
		\mathcal{D}[u] = -\frac{d}{dt}\mathcal{E}[u|u_\infty]
	\end{equation*}
	along the trajectory of \eqref{mass-action-system}. 
	The following lemma shows that the $L^1$ norm can be bounded by the relative entropy.
	\begin{lemma}\label{L1-bound}
		Assume that $\mathcal{E}[u|u_{\infty}]<+\infty$. Then 
		\begin{equation*}
		\|u_i\|_{L^1(\Omega)} = \overline{u_i} \leq \widetilde{K} \coleq 2\left(\mathcal{E}(u|u_{\infty}) + \sum_{i=1}^{I}u_{i,\infty}\right)
		\end{equation*}
		for $i=1,\ldots, I$.
	\end{lemma}	
	\begin{proof}
		The elementary inequalities $x\log(x/y) - x + y \geq (\sqrt{x} - \sqrt{y})^2 \geq \frac 12 x - y$ ensure that 
		\begin{equation*}
		\mathcal{E}[u|u_{\infty}]  = \sum_{i=1}^{I}\int_{\Omega}\left(u_i\log\frac{u_i}{u_{i,\infty}} - u_i + u_{i,\infty}\right) dx \geq \frac 12 \sum_{i=1}^{I}\int_{\Omega}u_i \, dx - \sum_{i=1}^{I}u_{i,\infty}.
		\end{equation*}
		Taking the non-negativity of the solution $u$ into account allows to conclude.
	\end{proof}

\begin{lemma}[Csisz\'ar--Kullback--Pinsker type inequality]\cite[Lemma 2.3]{FT17}\label{CKP}
	Fix a non-negative mass $M \in \mathbb R_{+}^m$. For all measurable functions $u: \Omega \to \mathbb R_+^I$ satisfying the mass conservation $\mathbb Q\, \overline{u} = M$, the following inequality holds,
	\begin{equation*}
	\mathcal{E}[u|u_{\infty}] \geq C_{CKP}\sum_{i=1}^{I}\|u_i - u_{i,\infty}\|_{L^1(\Omega)}^2,
	\end{equation*}
	in which $C_{CKP} > 0$ is a constant depending only on $\Omega$ and $u_{\infty}$.
\end{lemma}

The following two lemmas provide generalisations of the classical logarithmic Sobolev inequality, which are suited for non-linear diffusion.
\begin{lemma}[$L^m$-logarithmic Sobolev inequality]\label{Lp-LSI}\cite{DGGW08}
	For any $m \geq 1$, there exists a constant $C(\Omega,m)>0$ such that we have
	\begin{equation*}
	\int_{\Omega}\left|\nabla\left(u^{\frac{m}{2}} \right)\right|^2 dx \geq  C(\Omega,m)\left(\int_{\Omega}u\log\frac{u}{\overline{u}} \, dx\right)^{m}.
	\end{equation*}
\end{lemma}
\begin{lemma}[Generalised logarithmic Sobolev inequality]\label{Generalised-LSI}\cite{MM16} For any $m> (d-2)_+/d$, there exists a constant $C(\Omega, m)>0$ such that
	\begin{equation*}
	\int_{\Omega}u^{m-2}|\nabla u|^2 \, dx \geq C(\Omega,m)\overline{u}^{m-1}\int_{\Omega}u\log{\frac{u}{\overline{u}}} \, dx.
	\end{equation*}
\end{lemma}
The cornerstone of the entropy method is the entropy--entropy dissipation inequality which is proved in the following lemma.
\begin{lemma}\label{algebraic}
	Let $K>0$ and $m_i > (d-2)_+/d$ for all $1 \leq i \leq I$. There exist constants $C>0$ and $\alpha \geq 1$ depending on $K$ such that for any measurable function $u: \Omega \to \mathbb R_+^I$ subject to $\mathbb Q\overline{u} = \mathbb Q u_\infty$ and $\mathcal{E}[u|u_\infty]\leq K$, we have
	\begin{equation*}
	\mathcal{D}[u] \geq C(\mathcal E[u|u_\infty])^{\alpha}.
	\end{equation*}
\end{lemma}
\begin{proof}
	First, similarly to Lemma \ref{L1-bound}, we have the following $L^1$ bounds:
	\begin{equation*}
	\overline{u}_i = \|u_i\|_{L^1(\Omega)} \leq M_1\coleq 2\left(K + \sum_{i=1}^{I}u_{i,\infty}\right).
	\end{equation*}
	For all $i$ with $(d-2)_+/d < m_i < 1$, we apply Lemma \ref{Generalised-LSI} to have
	\begin{equation}\label{fast-diffusion}
	\begin{aligned}
	\int_{\Omega}u_i^{m_i-2}|\nabla u_i|^2 \, dx&\geq C(\Omega,m_i)\overline{u}_i^{m_i - 1}\int_{\Omega}u_i\log{\frac{u_i}{\overline{u}_i}} \, dx\\
	&\geq C(\Omega,m_i)M_1^{m_i - 1}\int_{\Omega}u_i\log{\frac{u_i}{\overline{u}_i}} \, dx
	\end{aligned}
	\end{equation}
	since $m_i < 1$ and $\overline{u}_i \leq M_1$. Therefore, by applying Lemma \ref{Lp-LSI}, we can estimate the entropy dissipation from below as follows: 
	\begin{equation*}
	\mathcal D[u] \geq K_1\left[\sum_{i=1}^{I}\left(\int_{\Omega}u_i\log{\frac{u_i}{\overline{u}_i}} \, dx\right)^{\max\{1,m_i\}} + \sum_{r=1}^{R}k_ru_{\infty}^{y_r}\int_{\Omega}\Psi\left(\frac{u^{y_r}}{u_{\infty}^{y_r}};\frac{u^{y_r'}}{u_{\infty}^{y_r'}}\right) dx\right].
	\end{equation*}
	For simplicity, we denote $\alpha = \max_{i=1,\ldots, I}\{1, m_i\}$,  
	\begin{equation*}
	D_i = \int_{\Omega}u_i\log{\frac{u_i}{\overline{u}_i}} \, dx \quad \text{ for } i = 1,\ldots, I, \quad \text{ and } \quad F = \sum_{r=1}^{R}k_ru_{\infty}^{y_r}\int_{\Omega}\Psi\left(\frac{u^{y_r}}{u_{\infty}^{y_r}};\frac{u^{y_r'}}{u_{\infty}^{y_r'}}\right) dx.
	\end{equation*}
	If $D_i \geq 1$ for some $i \in \{1,\ldots, I\}$ or $F\geq 1$, then we have
	\begin{equation}\label{alpha1}
	\mathcal D[u] \geq K_1 \geq \frac{K_1}{K^{\alpha}}\left(\mathcal E[u|u_{\infty}]\right)^{\alpha}
	\end{equation}
	using $\mathcal E[u|u_{\infty}]\leq K$. It remains to consider the case $D_i \leq 1$ for all $i=1,\ldots, I$ and $F \leq 1$. Here, we find   
	\begin{equation}\label{alpha1_0}
	\mathcal{D}[u]\geq K_1\left[\sum_{i=1}^{I}D_i^{\max\{1,m_i\}} + F\right] \geq K_1\left[\sum_{i=1}^{I}D_i^{\alpha} + F^{\alpha} \right] \geq K_2\left[\sum_{i=1}^{I}D_i + F\right]^{\alpha}.
	\end{equation}
	By applying now the entropy--entropy production estimate for the case of linear diffusion (cf.\ \cite[Theorem 1.1]{FT18}), we have
	\begin{equation}\label{alpha2}
	\sum_{i=1}^{I}D_i + F \geq K_3\mathcal E[u|u_{\infty}]
	\end{equation}
	and, thus,
	\begin{equation*}
	\mathcal D[u] \geq K_4(\mathcal E[u|u_{\infty}])^{\alpha}.
	\end{equation*}
	From \eqref{alpha1} and \eqref{alpha2}, we obtain the desired estimate.		
	\end{proof}
	
	\subsection{Convergence of renormalised solutions}
	\begin{proof}[Proof of Theorem \ref{thm2}]
		The global existence of renormalised solutions for \eqref{mass-action-system} follows from Theorem \ref{existence-renormalised}. Indeed, the assumptions \eqref{F1} and \eqref{F2} can be easily checked using \eqref{reaction-term}, and \eqref{F3} follows from the identity (cf.\ \cite[Proposition 2.1]{DFT17})
		\begin{equation*}
			\sum_{i=1}^{I}f_i(u)(-\log u_{i,\infty} + \log u_i) = -\sum_{r=1}^{R}k_ru_{\infty}^{y_r}\Psi\left(\frac{u^{y_r}}{u_{\infty}^{y_r}};\frac{u^{y_r'}}{u_{\infty}^{y_r'}}\right)\leq 0
		\end{equation*}
		recalling $\Psi(x;y) = x\log(x/y) - x + y$.		

		\medskip
		Thanks to \eqref{conservation_law} and \eqref{entropy_law}, we have the conservation laws
		\begin{equation*}
			\mathbb Q\overline{u(t)} = \mathbb Q \overline{u_0} \quad \text{ for all } \quad t\geq 0,
		\end{equation*}		
		and
		\begin{equation*}
			\mathcal E[u(t)|u_\infty] + \int_s^t\mathcal{D}[u(\tau)]d\tau \leq \mathcal E[u(s)|u_\infty] \quad \text{ for all } \quad 0 \leq s \leq t.
		\end{equation*}
		We now apply Lemma \ref{algebraic} to get
			\begin{equation}\label{d1}
				\mathcal E[u(t)|u_{\infty}] + C\int_{s}^{t}\left(\mathcal  E[u(\tau)|u_{\infty}]\right)^{\alpha}d\tau \leq \mathcal E[u(s)|u_{\infty}] \quad \text{ for all } \quad 0 \leq s \leq t.
			\end{equation}
			If $\alpha = 1$, we immediately get the exponential convergence
			\begin{equation*}
				\mathcal E[u(t)|u_{\infty}] \leq e^{-\lambda t}\mathcal E[u_0|u_{\infty}] \quad \text{ for all } \quad t\geq 0
			\end{equation*}
			for some constant $\lambda >0$. For $\alpha > 1$, we first obtain the algebraic decay 
			\begin{equation}\label{d2}
				\mathcal E[u(t)|u_{\infty}] \leq \frac{1}{\left(\mathcal E[u_0|u_{\infty}]^{1-\alpha} + C(\alpha-1)t \right)^{1/(\alpha-1)}}.
			\end{equation}
			Indeed, we define 
			\begin{equation*}
				\psi(t) = \E[u(t)|u_\infty] \quad \text{ and } \quad \varphi(s) = \int_s^t\bra{\E[u(\tau)|u_\infty]}^{\alpha}d\tau.
			\end{equation*}
			It follows from \eqref{d1} that
			\begin{equation*}
				\psi(t) + C\varphi(s) \leq \psi(s).
			\end{equation*}
			Differentiating $\varphi(s)$ with respect to $s$ gives
			\begin{equation*}
				\frac{d}{ds}\varphi(s) = -(\psi(s))^{\alpha} \leq -\bra{\psi(t) + C\varphi(s)}^{\alpha}
			\end{equation*}
			and consequently
			\begin{equation*}
				\frac{d}{ds}\sbra{\psi(t) + C\varphi(s)} + C\sbra{\psi(t) + C\varphi(s)}^{\alpha} \leq 0.
			\end{equation*}
			Applying a non-linear Gronwall inequality yields
			\begin{equation*}
				\psi(t) + C\varphi(s) \leq \frac{1}{\sbra{\bra{\psi(t)+C\varphi(0)}^{1-\alpha} + C(\alpha-1)s}^{1/(\alpha-1)}} \leq \frac{1}{\sbra{\psi(0)^{1-\alpha} + C(\alpha-1)s}^{1/(\alpha-1)}}
			\end{equation*}
			thanks to $\psi(t) + C\varphi(0) \leq \psi(0)$ and $\alpha>1$. Setting $s = t$ and using $\varphi(t) = 0$ leads us to the desired estimate \eqref{d2}.
			Owing to the Csisz\'ar--Kullback--Pinsker inequality, it follows
			\begin{equation*}
				\sum_{i=1}^{I}\|u_i(t) - u_{i,\infty}\|_{L^1(\Omega)}^2 \leq \frac{C_{CKP}^{-1}}{\left(\mathcal E[u_0|u_{\infty}]^{1-\alpha} + C(\alpha-1)t \right)^{1/(\alpha-1)}}.
			\end{equation*}
			Therefore, there exists an explicit $T_0>0$ such that 
			\begin{equation*}
				\|u_i(t) - u_{i,\infty}\|_{L^1(\Omega)} \leq \frac 12 u_{i,\infty} \quad \text{ for all } \quad t\geq T_0
			\end{equation*}
			and, thus,
			\begin{equation*}
				\overline{u}_i(t) = \|u_i(t)\|_{L^1(\Omega)} \geq \frac 12 u_{i,\infty} \quad \text{ for all } \quad t\geq T_0.
			\end{equation*}
			Using this property, we estimate $\mathcal D[u(t)]$ for $t\geq T_0$ as follows. If $m_i<1$, then it is similar to \eqref{fast-diffusion} that 
			\begin{equation}\label{k1}
			\begin{aligned}
			\int_{\Omega}\overline{u}_i(t)^{m_i-2}|\nabla u_i(t)|^2 \, dx &\geq C(\Omega,m_i)\overline{u}_i(t)^{m_i - 1}\int_{\Omega}u_i(t)\log{\frac{u_i(t)}{\overline{u}_i(t)}} \, dx\\ &\geq C(\Omega,m_i)\tilde{K}^{m_i - 1}\int_{\Omega}u_i(t)\log{\frac{u_i(t)}{\overline{u}_i(t)}} \, dx,
			\end{aligned}
			\end{equation}
			and for $m_i\geq 1$,
			\begin{equation}\label{k2}
			\begin{aligned}
			\int_{\Omega}\overline{u}_i(t)^{m_i-2}|\nabla u_i(t)|^2 \, dx &\geq C(\Omega,m_i)\overline{u}_i(t)^{m_i - 1}\int_{\Omega}u_i(t)\log{\frac{u_i(t)}{\overline{u}_i(t)}} \, dx\\ &\geq C(\Omega,m_i)\left(\frac{u_{i,\infty}}{2}\right)^{m_i - 1}\int_{\Omega}u_i(t)\log{\frac{u_i(t)}{\overline{u}_i(t)}} \, dx
			\end{aligned}
			\end{equation}
			since $\overline{u}_i(t) \geq \frac 12 u_{i,\infty}$ for all $t\geq T_0$. Therefore, for all $t\geq T_0$, we can estimate $\mathcal D[u(t)]$ as
			\begin{equation*}
				\mathcal D[u(t)] \geq K_5\left[\sum_{i=1}^{I}\int_{\Omega}{u}_i(t)\log\frac{u_i(t)}{\overline{u}_i(t)} \, dx +  \sum_{r=1}^{R}k_ru_{\infty}^{y_r}\int_{\Omega}\Psi\left(\frac{u(t)^{y_r}}{u_{\infty}^{y_r}};\frac{u(t)^{y_r'}}{u_{\infty}^{y_r'}}\right) dx\right].
			\end{equation*}
			Using again the entropy--entropy production estimate in the case of linear diffusion (cf.\ \cite[Theorem 1.1]{FT18}), we have
			\begin{equation*}
				\mathcal D[u(t)] \geq K_6\mathcal E[u(t)|u_{\infty}] \quad \text{ for all } \quad t\geq T_0,
			\end{equation*}
			which leads to the exponential convergence
			\begin{equation*}
				\mathcal E[u(t)|u_{\infty}] \leq e^{-K_6(t-T_0)}\mathcal E[u(T_0)|u_{\infty}] \quad \text{ for all } \quad t\geq T_0.
			\end{equation*}
			Since $T_0$ can be explicitly computed, we in fact get the exponential convergence for all $t\geq 0$, i.e.
			\begin{equation*}
				\mathcal E[u(t)|u_{\infty}] \leq Ce^{-\lambda t}\mathcal E[u_0|u_{\infty}] \quad \text{ for all } \quad t\geq 0,
			\end{equation*}
			for some constants $C>0$ and $\lambda >0$. Thanks to the Csisz\'ar--Kullback--Pinsker inequality, we finally obtain the desired exponential convergence to equilibrium.
	\end{proof}
	\subsection{Convergence of weak and bounded solutions}
	We consider the approximate $\ue$ to the regularised system \eqref{approx-system}. From Lemma \ref{a.e.-limit}, we know that, up to a subsequence, 
	\begin{equation}\label{ae}
		\ue_i \xrightarrow{\eps \to 0} u_i \quad \text{a.e.\ in } Q_T.
	\end{equation}
	The next lemma establishes uniform-in-$\eps$ {\it a priori} bounds.
	\begin{lemma}\label{add_lem0}
		Under the assumptions \eqref{F1}--\eqref{F3} and $0\leq u_0$ with $u_{i,0}\log u_{i,0}\in L^2(\Omega)$ for all $i=1,\ldots, I$, there exists a constant $C_T>0$ which depends on $T>0$ but not on $\eps>0$, such that
		\begin{equation}\label{uniform-bound}
			\int_{Q_T}(\log \ue_i)^2(\ue_i)^{m_i+1}\, dx \, dt \leq C_T \quad \text{ for all } \quad i=1,\ldots, I.
		\end{equation}
	\end{lemma}
	\begin{proof}
		The proof is similar to \cite[Proof of Theorem 2.7]{LP17}. For the sake of completeness, we present it in Appendix \ref{app2}.
	\end{proof}
	\begin{proof}[Proof of Theorem \ref{thm3}]
		By testing \eqref{approx-system} with smooth test functions $\psi \in C^\infty([0,T];C_0^\infty(\Omega))$ satisfying $\psi(\cdot, T) = 0$, we have
		\begin{equation}
			-\int_{Q_T}\ue_i \partial_t \psi \, dx \, dt - \int_{Q_T}d_i(\ue_i)^{m_i}\Delta \psi \, dx \, dt = -\int_{\Omega}u_{i,0}^{\eps}\psi(\cdot,0) \, dx + \int_{Q_T}f_i^\eps(\ue)\psi \, dx \, dt
		\end{equation}
		It follows from \eqref{ae} that $(\ue_i)^{m_i} \to u_i^{m_i}$ a.e.\ in $Q_T$. Due to the bound \eqref{uniform-bound}, it holds that the set $\{(\ue_i)^{m_i}\}_{\eps>0}$ is uniformly integrable in $Q_T$. Thus, thanks to Vitali's lemma, we deduce that $(\ue_i)^{m_i}$ converges strongly to $u_i^{m_i}$ in $L^1(Q_T)$, which implies
		\begin{equation*}
			\int_{Q_T}d_i(\ue_i)^{m_i}\Delta \psi \, dx \, dt \to \int_{Q_T}d_iu_i^{m_i}\Delta \psi \, dx \, dt.
		\end{equation*}
		From the definition of $f_i^\eps$ in \eqref{approx-system} as well as \eqref{ae}, we have
		\begin{equation}\label{ae_f}
			f_i^\eps(\ue) \to f_i(u) \quad \text{ a.e.\ in\ } Q_T.
		\end{equation}
		It remains to show that the set $\{f_i^\eps(\ue)\}_{\eps>0}$ is uniformly integrable. Indeed, let $K \subset Q_T$ be a measurable and compact set with its measure being denoted by $|K|$. From \eqref{mu} and  \eqref{high_mi}, we have
		\begin{equation}
		\label{c3}
		\begin{aligned}
			\int_{K}|f_i^\eps(\ue)|\, dx \, dt &\leq \int_{K}|f_i(\ue)|\, dx \, dt \leq C\int_{K}\bra{1+\sumi|\ue_i|^{\mu}}\, dx \, dt\\
			&\leq C\sbra{|K| + \int_{K}\sumi|\ue_i|^{m_i+1}\, dx \, dt}.
		\end{aligned}
		\end{equation}
		It is easy to see that for any $\delta>0$, there exists some $C_{\delta}>0$ such that for all $x\geq 0$,
		\begin{equation*}
			x^{m_i+1} \leq \delta(\log x)^2x^{m_i+1} +  C_{\delta}.
		\end{equation*}
		It follows from \eqref{c3} and Lemma \ref{add_lem0} that
		\begin{equation*}
		\begin{aligned}
			\int_{K}|f_i^\eps(\ue)|\, dx \, dt &\leq C\sbra{|K| + \delta \int_{K}\sumi(\log \ue_i)^2(\ue_i)^{m_i+1} + IC_\delta |K| }\\
			&\leq C_{\delta}|K| + C\delta.
		\end{aligned}
		\end{equation*}
		Therefore, for any $\tau > 0$, we first choose $\delta>0$ fixed such that $C\delta < \tau/2$ and then $K \subset Q_T$ arbitrarily according to $C_\delta|K| < \tau/2$ to ultimately obtain
		\begin{equation*}
			\int_{K}|f_i^\eps(\ue)|\, dx \, dt \leq \tau,
		\end{equation*}
		which is exactly the uniform integrability of $\{f_i^\eps(\ue)\}_{\eps>0}$. From this and \eqref{ae_f}, we can apply Vitali's lemma again to get $f_i^\eps(\ue) \to f_i(u)$ strongly in $L^1(Q_T)$ and, therefore, 
		\begin{equation*}
			\int_{Q_T}f_i^\eps(\ue)\psi \, dx \, dt \to \int_{Q_T}f_i(u)\psi \, dx \, dt.
		\end{equation*}
		This gives us the existence of a global weak solution. The exponential convergence to equilibrium in $L^1(\Omega)$ follows from Theorem \ref{thm2}.
		
		\medskip
		Concerning bounded solutions, we first obtain from Lemma \ref{add_lem0} that 
		\begin{equation*}
			\|u_i\|_{L^{m_i+1}(Q_T)} \leq C_T \quad \text{ for all } \quad i=1,\ldots, I.
		\end{equation*}
		Under the assumptions \eqref{stronger_porous} on the porous medium exponent, we can apply \cite[Lemma 2.2]{FLT20} to get
		\begin{equation*}
			\|u_i\|_{L^\infty(Q_T)} \leq C_T\quad \text{ for all } \quad i=1,\ldots, I.
		\end{equation*}
		Note that the constant $C_T$ grows at most polynomially in $T$. This implies the exponential convergence to equilibrium in $L^p(\Omega)$ for any $1\leq p < \infty$ in \eqref{Lp-convergence}. Indeed, from the exponential convergence in $L^1(\Omega)$, we can use the interpolation inequality to have 
		\begin{align*}
			\|u_i(T) - u_{i,\infty}\|_{L^p(\Omega)} &\leq \|u_i(T) - u_{i,\infty}\|_{L^\infty(\Omega)}^{1-1/p}\|u_i(T) - u_{i,\infty}\|_{L^1(\Omega)}^{1/p} \\
			&\leq C_T^{1-1/p} \big(Ce^{-\lambda T}\big)^{1/p} \leq C_pe^{-\lambda_p T}
		\end{align*}
		for some $0<\lambda_p < \lambda$, since $C_T$ grows at most polynomially in $T$.
	\end{proof}
	\subsection{Uniform convergence of approximate solutions}
		We remark that all the involved constants in the following results {\it do not depend on $\eps$}.
		\begin{lemma}\label{lem0}
			There exists a constant $L_0>0$ such that 
			\begin{equation}\label{L1-bound-1}
				\sup_{t\geq 0}\max_{i=1,\ldots, I}\|\ue_i(t)\|_{L^1(\Omega)} \leq L_0.
			\end{equation}
		\end{lemma}
		\begin{proof}
			It follows from
			\begin{equation*}
				\frac{d}{dt}\E[\ue|u_\infty] = -\D[\ue] \leq 0
			\end{equation*}
			that 
			\begin{equation}\label{bound_entropy}
				\E[\ue(t)|u_\infty] \leq \E[\ue_0|u_{\infty}].
			\end{equation}
			Thanks to \eqref{assump1}, $\sup_{\eps>0}\max_{i=1,\ldots, I}\E[\ue_0|u_{\infty}] < +\infty$. 		The $L^1$ bound \eqref{L1-bound-1} follows with the same arguments as in Lemma \ref{L1-bound}. 
		\end{proof}
		\begin{lemma}\label{lem2}
			There exists a constant $L_2>0$ such that
			\begin{multline*}
				\sum_{i=1}^I\|\sqrt{u_i^\eps} - \ol{\sqrt{u_i^\eps}}\|_{L^2(\Omega)}^2 + \sum_{r=1}^R\int_{\Omega}\frac{1}{1+\eps|f(\ue)|}\sbra{\sqrt{\frac{\ue}{u_\infty}}^{y_r} - \sqrt{\frac{\ue}{u_\infty}}^{y_r'}}^2dx\\
				\geq L_2\sum_{r=1}^{R}\sbra{\frac{\ol{\sqrt {\ue}}^{y_r}}{\sqrt{u_\infty}^{y_r}} -  \frac{\ol{\sqrt {\ue}}^{y_r'}}{\sqrt{u_\infty}^{y_r'}}}^2.
			\end{multline*}
		\end{lemma}
		\begin{proof}
			We introduce the short hand notation $$\ve_i = \sqrt{\ue_i}, \quad \ve = (\ve_i)_{i=1,\ldots, I}, \quad v_{i,\infty} = \sqrt{u_{i,\infty}}, \quad v_\infty = (v_{i,\infty})_{i=1,\ldots, I},$$
			and $g(\ve) = f((\ve)^2) = f(\ue)$. Moreover, for each $i\in \{1,\ldots, I\}$, we set
			\begin{equation*}
				\ve_i(x) = \ol{\ve_i}  + \delta_i(x), \quad x\in\Omega.
			\end{equation*}
			The desired inequality in Lemma \ref{lem2} becomes
			\begin{equation}\label{new_ineq}
				\sum_{i=1}^{I}\|\delta_i\|_{L^2(\Omega)}^2 + \sum_{r=1}^{R}\int_{\Omega}\frac{1}{1+\eps|g(\ve)|}\sbra{\frac{(\ve)^{y_r}}{v_{\infty}^{y_r}} - \frac{(\ve)^{y_r'}}{v_{\infty}^{y_r'}}}^2dx
				\geq L_2\sum_{r=1}^{R}\sbra{\frac{\ol{\ve}^{y_r}}{v_\infty^{y_r}} - \frac{\ol{\ve}^{y_r'}}{v_\infty^{y_r'}}}^2.
			\end{equation}
			Using Lemma \ref{lem0}, we observe by H\"older's inequality, noting that $|\Omega| = 1$, that 
			\begin{equation*}
				\ol{\ve_i} = \ol{\sqrt{\ue_i}} \leq \sqrt{\ol{\ue_i}} \leq \sqrt{L_0}
			\end{equation*}
			and
			\begin{equation*}
				\|\delta_i\|_{L^2(\Omega)} = \sqrt{\ol{(\ve_i)^2} - \ol{\ve_i}^2} \leq \sqrt{\ol{\ue_i}} \leq \sqrt{L_0}.
			\end{equation*}
			Therefore, there exists a constant $K_7>0$ such that 
			\begin{equation}\label{1_lem2}
				\sum_{r=1}^{R}\sbra{\frac{\ol{\sqrt {\ue}}^{y_r}}{\sqrt{u_\infty}^{y_r}} -  \frac{\ol{\sqrt {\ue}}^{y_r'}}{\sqrt{u_\infty}^{y_r'}}}^2 \leq K_7.
			\end{equation}
			We decompose $\Omega$ into $\Omega = \Omega_1 \cup \Omega_2$ where
			\begin{equation*}
			\Omega_1 = \{x\in\Omega\,:\, |\delta_i(x)| \leq 1 \quad\forall i=1,\ldots, I\} \quad \text{and} \quad \Omega_2 = \Omega\backslash\Omega_1.
			\end{equation*}
			Note that on $\Omega_1$ we have
			\begin{equation*}
				0\leq \ve_i(x) = \ol{\ve_i} + \delta_i \leq \sqrt{L_0} + 1 \quad \text{ for all } \quad i=1,\ldots, I.
			\end{equation*}
			Therefore, there exists a constant $K_8>0$ such that
			\begin{equation}\label{2_lem2}
				\frac{1}{1+\eps|g(\ve(x))|} \geq K_8 \quad \text{ for all } \quad x\in\Omega_1.
			\end{equation}
			By using Taylor's expansion, we have for all $x\in\Omega_1$,
			\begin{align}
				\sbra{\frac{(\ve(x))^{y_r}}{v_{\infty}^{y_r}} - \frac{(\ve(x))^{y_r}}{v_{\infty}^{y_r}}}^2&= \sbra{\frac{1}{v_\infty^{y_r}}\prod_{i=1}^{I}\sbra{\ol{\ve_i} + \delta_i(x)}^{y_{r,i}} - \frac{1}{v_\infty^{y_r'}}\prod_{i=1}^{I}\sbra{\ol{\ve_i} + \delta_i(x)}^{y_{r,i}'}}^2 \nonumber \\
				&=\sbra{\frac{1}{v_\infty^{y_r}}\prod_{i=1}^{I}\bra{\ol{\ve_i}^{y_{r,i}} + R_1\delta_i(x)} - \frac{1}{v_\infty^{y_r'}}\prod_{i=1}^{I}\bra{\ol{\ve_i}^{y_{r,i}'} + R_2\delta_i(x)}}^2 \nonumber \\
				&\geq \frac 12\sbra{\frac{\ol{\ve}^{y_r}}{v_{\infty}^{y_r}} - \frac{\ol{\ve}^{y_r'}}{v_{\infty}^{y_r'}}}^2 - K_9\sum_{i=1}^{I}|\delta_i(x)|^2 \label{3_lem2}
			\end{align}
			for some $K_9>0$. From \eqref{2_lem2} and \eqref{3_lem2}, we can estimate the second sum on the left hand side of \eqref{new_ineq} as
			\begin{equation}\label{4_lem2}
			\begin{aligned}
				&\sum_{r=1}^{R}\int_{\Omega}\frac{1}{1+\eps|g(\ve)|}\sbra{\frac{(\ve)^{y_r}}{v_{\infty}^{y_r}} - \frac{(\ve)^{y_r'}}{v_{\infty}^{y_r'}}}^2dx\\
				&\qquad \geq \sum_{r=1}^{R}\int_{\Omega_1}\frac{1}{1+\eps|g(\ve)|}\sbra{\frac{(\ve)^{y_r}}{v_{\infty}^{y_r}} - \frac{(\ve)^{y_r'}}{v_{\infty}^{y_r'}}}^2dx\\
				&\qquad \geq K_8\sum_{r=1}^{R}\int_{\Omega_1}\bra{\frac 12\sbra{\frac{\ol{\ve}^{y_r}}{v_{\infty}^{y_r}} - \frac{\ol{\ve}^{y_r'}}{v_{\infty}^{y_r'}}}^2 - K_9\sum_{i=1}^{I}|\delta_i(x)|^2}dx\\
				&\qquad \geq \frac{K_8}{2}|\Omega_1|\sum_{r=1}^{R}\sbra{\frac{\ol{\ve}^{y_r}}{v_{\infty}^{y_r}} - \frac{\ol{\ve}^{y_r'}}{v_{\infty}^{y_r'}}}^2 - K_8K_9R\sum_{i=1}^{I}\|\delta_i\|_{L^2(\Omega)}^2.
			\end{aligned}
			\end{equation}
			On the other hand, the first sum on the left hand side of \eqref{new_ineq} is estimated as
			\begin{equation}\label{5_lem2}
				\begin{aligned}
				\sum_{i=1}^{I}\|\delta_i\|_{L^2(\Omega)}^2 &\geq \frac 12\sum_{i=1}^{I}\|\delta_i\|_{L^2(\Omega)}^2 + \frac 12\int_{\Omega_2}\sum_{i=1}^{I}|\delta_i(x)|^2dx\\
				&\geq \frac 12\sum_{i=1}^{I}\|\delta_i\|_{L^2(\Omega)}^2 + \frac 12 |\Omega_2|\\
				&\geq \frac 12\sum_{i=1}^{I}\|\delta_i\|_{L^2(\Omega)}^2 + \frac{1}{2K_7}|\Omega_2|\sum_{r=1}^{R}\sbra{\frac{\ol{\ve}^{y_r}}{v_{\infty}^{y_r}} - \frac{\ol{\ve}^{y_r'}}{v_{\infty}^{y_r'}}}^2.
				\end{aligned}
			\end{equation}
			By combining \eqref{4_lem2} and \eqref{5_lem2}, we have for any $\theta \in (0,1]$,
			\begin{equation}
			\begin{aligned}
				\text{LHS of (\ref{new_ineq})} &\geq \bra{\theta\frac{K_8}{2}|\Omega_1|+\frac{1}{2K_7}|\Omega_2|}\sum_{r=1}^{R}\sbra{\frac{\ol{\ve}^{y_r}}{v_{\infty}^{y_r}} - \frac{\ol{\ve}^{y_r'}}{v_{\infty}^{y_r'}}}^2\\
				&\quad + \bra{\frac 12 - \theta K_8K_9 R}\sum_{i=1}^{I}\|\delta_i\|_{L^2(\Omega)}^2.
			\end{aligned}
			\end{equation}
			Thus, by choosing $\theta = \min\left\{1; \frac{1}{2K_8K_9R} \right\}$, we finally obtain the desired inequality \eqref{new_ineq} with $L_2 = \min\left\{\frac{\theta K_8}{2}; \frac{1}{2K_7}\right\}$.
		\end{proof}
		\begin{lemma}\label{lem3}
			There exists a constant $L_3>0$ such that
			\begin{align*}
				\sum_{i=1}^I\|\sqrt{u_i^\eps} - \ol{\sqrt{u_i^\eps}}\|_{L^2(\Omega)}^2 + \sum_{r=1}^{R}\sbra{\frac{\ol{\sqrt {\ue}}^{y_r}}{\sqrt{u_\infty}^{y_r}} -  \frac{\ol{\sqrt {\ue}}^{y_r'}}{\sqrt{u_\infty}^{y_r'}}}^2 \geq L_3\sum_{r=1}^{R}\sbra{\sqrt{\frac{\ol{\ue}}{u_\infty}}^{y_r} - \sqrt{\frac{\ol{\ue}}{u_\infty}}^{y_r'}}^2.
			\end{align*}
		\end{lemma}
		\begin{proof}
			The proof of this lemma follows by the same arguments as in \cite[Lemma 2.7]{FT18}, so we omit it here.
		\end{proof}
		We are now ready to prove Theorem \ref{thm4}.
		\begin{proof}[Proof of Theorem \ref{thm4}]
			We proceed similarly to the proof of Lemma \ref{algebraic}. If $(d-2)_+/d < m_i < 1$, we have, thanks to Lemma \ref{Generalised-LSI} and the bound \eqref{L1-bound-1},
			\begin{equation*}
				\int_{\Omega}(\ue_i)^{m_i-2}|\na \ue_i|^2 \, dx \geq C \ol{\ue_i}^{m_i-1}\int_{\Omega}\ue_i\log\frac{\ue_i}{\ol{\ue_i}} \, dx \geq C L_0^{m_i-1}\int_{\Omega}\ue_i\log\frac{\ue_i}{\ol{\ue_i}} \, dx
			\end{equation*}
			with a constant $C = C(\Omega,m_i) > 0$. If $m_i\geq 1$, we apply Lemma \ref{Lp-LSI} to get
			\begin{equation*}
				\int_{\Omega}(\ue_i)^{m_i-2}|\na \ue_i|^2 \, dx \geq C(\Omega,m_i)\bra{\int_{\Omega}\ue_i\log\frac{\ue_i}{\ol{\ue_i}} \, dx}^{m_i}.
			\end{equation*}
			Therefore, by using $\Psi(x,y) = x\log(x/y) - x + y \geq \bra{\sqrt x - \sqrt y}^2$ and noting that $$\int_{\Omega}\ue_i\log\frac{\ue_i}{\ol{\ue_i}} \, dx = \int_{\Omega}\bra{\ue_i\log\frac{\ue_i}{\ol{\ue_i}} - \ue_i + \ol{\ue_i}}dx \geq \|\sqrt{\ue_i} - \ol{\sqrt{\ue_i}}\|_{L^2(\Omega)}^2,$$ there exists a constant $K_{10}>0$ such that
			\begin{equation}\label{f1}
			\begin{aligned}
				\D[\ue]\geq K_{10}\biggl[&\sum_{i=1}^{I}\bra{\|\sqrt{\ue_i} - \ol{\sqrt{\ue_i}}\|_{L^2(\Omega)}^2}^{\max\{1, m_i\}}\\
				&\quad +  \sum_{r=1}^Rk_ru_\infty^{y_r}\int_{\Omega}\frac{1}{1+\eps|f(u^\eps)|}\biggl(\sqrt{\frac{u^\eps}{u_\infty}}^{y_r} -  \sqrt{\frac{u^\eps}{u_\infty}}^{y_r'}\biggr)^2 \, dx\biggr].
			\end{aligned}
			\end{equation}
			We define  
			\begin{equation*}
			\wh{D}_i:= \|\sqrt{\ue_i} - \ol{\sqrt{\ue_i}}\|_{L^2(\Omega)}^2, \quad \wh{F}:= \sum_{r=1}^Rk_ru_\infty^{y_r}\int_{\Omega}\frac{1}{1+\eps|f(u^\eps)|}\sbra{\sqrt{\frac{u^\eps}{u_\infty}}^{y_r} -  \sqrt{\frac{u^\eps}{u_\infty}}^{y_r'}}^2dx,
			\end{equation*}
			and we consider the following two cases.
			
			{\bf Case 1.}
			If $\wh{F} \geq 1$ or if there exists some $i\in \{1,\ldots, I\}$ such that $\wh{D}_i\geq 1$, we estimate from \eqref{f1} that 
			\begin{equation}\label{f2}
				\D[\ue] \geq K_{10} \geq K_{10}\frac{1}{\E[u_0|u_\infty]}\E[\ue|u_\infty].
			\end{equation}
			
			{\bf Case 2.}
			If $\wh{F}\leq 1$ and $\wh{D}_i\leq 1$ for all $i=1,\ldots, I$, then with $\alpha = \max_{i=1,\ldots,I}\{1,m_i\}$, we have from \eqref{f1} that 
			\begin{equation}\label{f3}
				\D[\ue] \geq K_{10}\sbra{\sum_{i=1}^{I}\wh{D}_i^{\alpha} + \wh{F}^{\alpha}} \geq K_{11}\sbra{\sum_{i=1}^{I}\wh{D}_i + \wh{F}}^{\alpha}
			\end{equation}
			for some $K_{11} = K_{11}(K_{10},\alpha)$. Thanks to Lemmas \ref{lem2} and \ref{lem3}, we have
			\begin{align}
				\sum_{i=1}^I\wh{D}_i + \wh{F} &\geq \min_{r=1,\ldots, R}\{1,k_ru_\infty^{y_r}\}\sbra{\frac 12\sum_{i=1}^I\|\sqrt{\ue_i} - \ol{\sqrt{\ue_i}}\|_{L^2(\Omega)}^2 + \frac{L_2}{2}\sum_{r=1}^{R}\!\bra{\!\frac{\ol{\sqrt {\ue}}^{y_r}}{\sqrt{u_\infty}^{y_r}} -  \frac{\ol{\sqrt {\ue}}^{y_r'}}{\sqrt{u_\infty}^{y_r'}}\!}^{\!\!2} }\nonumber\\
				&\geq K_{12}\sum_{r=1}^{R}\sbra{\sqrt{\frac{\ol{\ue}}{u_\infty}}^{y_r} - \sqrt{\frac{\ol{\ue}}{u_\infty}}^{y_r'}}^2\label{f4}
			\end{align}
			where
			\begin{equation*}
				K_{12} = \frac 12\min_{r=1,\ldots, R}\{L_2,1,k_ru_\infty^{y_r}\}L_3.
			\end{equation*}
			From here, we can use \cite[Lemma 2.5]{FT18} to estimate
			\begin{equation*}
				\E[\ue|u_\infty] \leq K_{13}\sum_{i=1}^{I}\bra{\sqrt{\frac{\ol{\ue_i}}{u_{i,\infty}}} - 1}^2,
			\end{equation*}
			and \cite[Eq. (11)]{FT18} to get
			\begin{equation}\label{f5}
				\sum_{r=1}^{R}\sbra{\sqrt{\frac{\ol{\ue}}{u_\infty}}^{y_r} - \sqrt{\frac{\ol{\ue}}{u_\infty}}^{y_r'}}^2 \geq K_{14}\sum_{i=1}^{I}\bra{\sqrt{\frac{\ol{\ue_i}}{u_{i,\infty}}} - 1}^2 \geq \frac{K_{14}}{K_{13}}\E[\ue|u_\infty].
			\end{equation}
			It follows from \eqref{f3}, \eqref{f4}, and \eqref{f5} that
			\begin{equation}\label{f6}
				\D[\ue] \geq K_{11}\bra{\frac{K_{12}K_{14}}{K_{13}}}^{\alpha}\E[\ue|u_\infty]^{\alpha}.
			\end{equation}
			
			\medskip
			Owing to \eqref{f2} and \eqref{f6}, there exists some $L_4>0$ such that
			\begin{equation*}
				\D[\ue] \geq L_4\E[\ue|u_\infty]^{\max\{1,\alpha\}}.
			\end{equation*}
			The rest of this proof follows exactly from that of Theorem \ref{thm2}, which helps us to eventually obtain the exponential convergence to equilibrium
			\begin{equation*}
				\sum_{i=1}^I\|\ue_i(t) - u_{i,\infty}\|_{L^1(\Omega)}^2 \leq Ce^{-\lambda t}
			\end{equation*}
			where $C, \lambda>0$ are {\it independent of $\eps$.}
		\end{proof}
	
\appendix
\section{Proof of the global existence of approximate solutions}\label{appendix}
In this appendix, we provide a proof for the existence of a global solution to the approximate system \eqref{approx-system}. While a similar proof, with Dirichlet instead of Neumann boundary conditions, can be found in \cite[Proof of Lemma 2.3]{LP17}, we mainly follow a standard Galerkin approach as presented in many textbooks.  

\medskip
{
For $0 < \delta \leq 1$ and $1 \leq i \leq I$, we consider the regularised system
\begin{equation}\label{bound}
\partial_t u_i^{\eps,\delta} - d_i\nabla\cdot\sbra{\bra{\frac{m_i|u_i^{\eps,\delta}|^{m_i-1}}{1+\delta m_i|u_i^{\eps,\delta}|^{m_i-1}}+\delta}\na u_i^{\eps,\delta}} = f_i^{\eps}(u^{\eps,\delta}),  \quad \nu \cdot \na u_i^{\eps,\delta} = 0,
\end{equation}
with initial data $u_i^{\eps,\delta}(x,0) = u_{i,0}^{\eps}(x) + \delta$, where we extend $f_i^\varepsilon$ to a locally Lipschitz continuous function on $\mathbb R^I$ by setting $f_i^\varepsilon(v) \coleq f_i^\varepsilon(\max\{v_1, 0\}, \dotsc, \max\{v_I, 0\})$. 
For the sake of readability, we subsequently write $w_i:= u_i^{\eps,\delta}$ and $H^{\delta}(w_i):= \frac{m_i |w_i|^{m_i-1}}{1+\delta m_i |w_i|^{m_i-1}} + \delta$. We employ the ansatz $w_i^k(x, t) = \sum_{j=1}^k \xi_i^{j,k}(t) e_j(x)$ with $k \in \mathbb N$, scalar coefficient functions $\xi_i^{j,k}(t)$, and the Schauder basis $e_j \in H^1(\Omega)$ of $L^2(\Omega)$ given by the orthonormal eigenfunctions of the Laplace operator satisfying $\nu \cdot \nabla e_j = 0$ on $\partial \Omega$. Along with the initial condition $\xi_i^{j,k}(0) = \int_\Omega (w_0)_i e_j \, dx$, the functions $\xi_i^{j,k}$, $1 \leq j \leq k$, are solutions to the set of equations 
\begin{align}
\label{eqregultestedej}
\int_\Omega \partial_t w_i^k e_j \, dx = -d_i \int_\Omega H^\delta(w_i^k) \nabla w_i^k \cdot \nabla e_j \, dx + \int_\Omega f_i^\varepsilon(w_i^k) e_j \, dx 
\end{align}
for $1 \leq j \leq k$, which can be recast into an ODE system for the time-dependent coefficient vector $\hat \xi_i^k (t) \coleq (\xi_i^{1,k} (t), \dotsc, \xi_i^{k,k} (t))$. As the right hand side continuously depends on $\hat \xi_i^k (t)$ via the bounded and continuous functions $H^\delta$ and $f_i^\varepsilon$, we know that a solution $\hat \xi_i^k (t) \in C^1([0,\tau], \mathbb R^I)$ exists for some $\tau > 0$. The energy estimate below, which is essentially obtained by multiplying \eqref{eqregultestedej} with $\xi_i^{j,k}(t)$ and summing over $j$, will indeed show that the solution exists for all times $t \geq 0$. Integrating \eqref{eqregultestedej} over $[0,t] \subset [0,\tau)$, we have 
\[
\| w_i^k(t) \|_{L^2(\Omega)}^2 + \delta d_i \int_0^t \| \nabla w_i^k(s) \|_{L^2(\Omega)}^2 \, ds \leq C \bigg( \| w_i^k(0) \|_{L^2(\Omega)}^2 + \int_0^t \int_\Omega \varepsilon^{-1} |w_i^k| \, dx \, ds \bigg), 
\]
where $C>0$ is independent of $k$ by the uniform bounds $H^\delta(w_i^k) \geq \delta$ and $f_i^\varepsilon(w_i^k) \leq \varepsilon^{-1}$. In order to control the $H^1(\Omega)$ norm of $w_i^k$, we first observe that $e_1$ is constant in space. Hence, \eqref{eqregultestedej} leads for all $t \in (0,\tau)$ to 
\[
\bigg| \int_\Omega w_i^k(x, t) \, dx \bigg| = \bigg| \int_\Omega w_i^k(x, 0) \, dx + \int_0^t \int_\Omega f_i^\varepsilon(w_i^k) \, dx \, ds \bigg| \leq \| w_i^k(0) \|_{L^1(\Omega)} + \frac{|\Omega|t}{\varepsilon}.
\]
Poincar\'e's inequality now yields 
\[
\| w_i^k(t) \|_{H^1(\Omega)}^2 
\leq C \bigg( \| \nabla w_i^k(t) \|_{L^2(\Omega)}^2 + \|w_i^k(0) \|_{L^2(\Omega)}^2 + \frac{|\Omega|^2\tau^2}{\varepsilon^2} \bigg).
\]
As a result, we arrive at 
\begin{align}
\label{eqregulbound}
\| w_i^k(t) \|_{L^2(\Omega)}^2 + \int_0^t \| w_i^k(s) \|_{H^1(\Omega)}^2 \, ds \leq C \bigg( 1 + \| w_i(0) \|_{L^2(\Omega)}^2 \bigg), 
\end{align}
where $C>0$ is independent of $k$, and which shows that the solution is bounded on $[0,\tau)$ and, thus, existing for all $t \geq 0$. Note that we applied the elementary estimate $\| w_i^k(0) \|_{L^2(\Omega)} \leq \| w_i(0) \|_{L^2(\Omega)}$ in the last step. 

\medskip

Let $P_k : L^2(\Omega) \rightarrow L^2(\Omega)$ be the orthogonal projection in $L^2(\Omega)$ onto $\mathrm{span}\{e_1, \dotsc, e_k\}$, and let $\phi \in L^2(0, T; H^1(\Omega))$ for some arbitrary but fixed $T > 0$. We recall the uniform bounds $H^\delta(w_i^k) \leq \delta^{-1} + \delta$ and $f_i^\varepsilon(w_i^k) \leq \varepsilon^{-1}$ leading to 
\begin{align*}
&\int_0^T \int_\Omega \partial_t w_i^k \phi \, dx \, dt = \int_0^T \int_\Omega \partial_t w_i^k P_k \phi \, dx \, dt \\ 
&\qquad = -d_i \int_0^T \int_\Omega H^\delta(w_i^k) \nabla w_i^k \cdot \nabla (P_k \phi) \, dx \, dt + \int_0^T \int_\Omega f_i^\varepsilon(w_i^k) P_k \phi \, dx \, dt \\ 
&\qquad \leq C \int_0^T \| \nabla w_i^k(t) \|_{L^2(\Omega)} \| \nabla (P_k \phi)(t) \|_{L^2(\Omega)} \, dt + C \int_0^T \| (P_k \phi)(t) \|_{L^2(\Omega)} \, dt \\
&\qquad \leq C( 1 + \| w_i^k \|_{L^2(0, T; H^1(\Omega))}) \| \phi \|_{L^2(0, T; H^1(\Omega))}
\end{align*}
with a constant $C > 0$ independent of $k$. We stress that the last estimate also employs the bound $\| (P_k \phi)(t) \|_{H^1(\Omega)} \leq \| \phi(t) \|_{H^1(\Omega)}$ for a.e.\ $t \in [0, T]$, which provides a uniform-in-$k$ bound for $\partial_t w_i^k \in L^2(0, T; (H^1(\Omega))')$ together with \eqref{eqregulbound}. 

\medskip 

We may choose a (not relabelled) subsequence $w_i^k$ being weakly converging for $k \rightarrow \infty$ to some $w_i$ in $L^2(0, T; H^1(\Omega))$ and in $H^1(0, T; (H^1(\Omega))')$. Strong convergence $w_i^k \rightarrow w_i$ in $L^2(0, T; L^2(\Omega)) \cong L^2(Q_T)$ now follows by means of the Aubin--Lions Lemma. Choosing another subsequence, we obtain for all $i=1,\ldots I$, $w_i^k \rightarrow w_i$ and, hence, $H^\delta(w_i^k) \rightarrow H^\delta(w_i)$ and $f_i^\varepsilon(w_i^k) \rightarrow f_i^\varepsilon(w_i)$ a.e.\ in $\Omega$ due to the continuity of $H^\delta$ and $f_i^\varepsilon$. The boundedness of $H^\delta$ and $f_i^\varepsilon$ further guarantees $H^\delta(w_i^k) \rightarrow H^\delta(w_i)$ and $f_i^\varepsilon(w_i^k) \rightarrow f_i^\varepsilon(w_i)$ strongly in $L^2(Q_T)$. This allows us to pass to the limit $k \rightarrow \infty$ in an integrated version of \eqref{eqregultestedej} resulting in  
\[
\int_\Omega w_i \phi \, dx \bigg|_0^T - \int_{Q_T} w_i \partial_t \phi \, dx \, dt + d_i \int_{Q_T} H^\delta(w_i) \nabla w_i \cdot \nabla \phi \, dx \, dt \\ = \int_{Q_T} f_i^\varepsilon(w_i) \phi \, dx \, dt
\]
for all $\phi \in C^\infty(\overline \Omega \times [0,T])$. By a density argument, \eqref{bound} holds as an identity in $L^2(0, T; (H^1(\Omega))')$. 
}

\medskip
Thanks to the quasi-positivity assumption \eqref{F2}, the solution $w_i$ is also non-negative. For $p> 1$, we have
\begin{multline}\label{bound_0}
\partial_t\|w_i\|_{L^p(\Omega)}^p + d_ip(p-1)\int_{\Omega}H^{\delta}(w_i)w_i^{p-2}|\na w_i|^2dx \\ = p\int_{\Omega}w_i^{p-1}f_i^{\eps}(w)dx
\leq p\|f_i^\eps(w)\|_{L^p(\Omega)}\left(\|w_i(t)\|_{L^p(\Omega)}^{p}\right)^{1-\frac 1p}.
\end{multline}
Applying the Gronwall lemma $y' \leq \alpha(t)y^{1-1/p} \Rightarrow y(t) \leq \bra{y(0)^{1/p} + \frac 1p\int_0^t\alpha(s)ds}^{p}$ for $y(t) = \|w_i(t)\|_{L^p(\Omega)}^p$ entails 
\begin{align*}
	\|w_i(t)\|_{L^p(\Omega)}^p&\leq \bra{\|w_i(0)\|_{L^p(\Omega)} + \int_0^t\|f_i^\eps(w(s))\|_{L^p(\Omega)}ds}^p\\
	&\leq \bra{\|w_i(0)\|_{L^p(\Omega)} + \|f_i^\eps(w)\|_{L^p(Q_T)}T^{\frac{p-1}{p}}}^p.
\end{align*}
Taking the root of order $p$ on both sides and letting $p\to \infty$, this leads to
\begin{equation*}
\sup_{t\in (0,T)}\|w_i(t)\|_{L^\infty(\Omega)} \leq \|w_i(0)\|_{L^{\infty}(\Omega)} + T\|f_i^\eps(w)\|_{L^{\infty}(Q_T)},
\end{equation*}
which means
\begin{equation}\label{bound_1}
\|w_i\|_{L^{\infty}(Q_T)} \leq C_{\eps,T} \quad \text{ for all } \quad i=1,\ldots, I,
\end{equation}
where $C_{\eps,T}$ is {\it independent} of $\delta>0$. Therefore, 
\begin{equation*}
H^{\delta}(w_i) = \frac{m_iw_i^{m_i-1}}{1+\delta m_iw_i^{m_i-1}} + \delta \geq \frac{m_iw_i^{m_i-1}}{1+m_iC_{\eps,T}^{m_i-1}} + \delta.
\end{equation*}
By integrating \eqref{bound_0} on $(0,T)$, and using \eqref{bound_1}, we get
\begin{equation}\label{bound_0_1}
\begin{aligned}
&m_i \int_0^T\int_{\Omega}w_i^{m_i-1}w_i^{p-2}|\na w_i|^2\, dx \, dt + \delta\int_0^T\int_{\Omega}w_i^{p-2}|\na w_i|^2\, dx \, dt\\
&\qquad \leq \frac{1}{d_i(p-1)}\bra{\|w_i(0)\|_{L^p(\Omega)}^p + \int_0^T\int_{\Omega}f_i^\eps(w)w_i^{p-1}\, dx \, dt}\\
&\qquad \leq \frac{1}{d_i(p-1)}\bra{\|w_i(0)\|_{L^p(\Omega)}^p + |\Omega| T \|f_i^\eps(w)\|_{L^\infty(Q_T)}C_{\eps,T}^{p-1}}, 
\end{aligned}
\end{equation}
which implies, in particular,
\begin{equation}\label{bound_2}
\Big\|\na  \Big(w_i^{\frac{m_i+p-1}{2}} \Big) \Big\|_{L^2(Q_T)} \leq C_{\eps,T}
\end{equation}
for all $p>1$. By testing the equation \eqref{bound} with $\phi \in L^2(0,T;H^1(\Omega))$ and keeping \eqref{bound_1} in mind, we see that $\{\partial_t w_i\}$ is bounded in $L^2(0,T;(H^1(\Omega))^{'})$ uniformly in $\delta>0$. Hence, from this, \eqref{bound_1}, and \eqref{bound_2}, we can apply a non-linear Aubin--Lions lemma, see e.g.\ \cite[Theorem~1]{moussa2016some}, to ensure the existence of a subsequence (not relabelled) such that $w_i = u_{i}^{\eps,\delta} \to u_i^{\eps}$ strongly in $L^2(Q_T)$ as $\delta\to 0$. From here onwards, we will write $u_i^{\eps,\delta}$ instead of $w_i$. Thanks to the $L^\infty$ bound \eqref{bound_1}, this convergence in fact holds in $L^p(Q_T)$ for any $1\leq p<\infty$. It remains to pass to the limit $\delta \to 0$ in the weak formulation of \eqref{bound} with $\psi \in L^2(0,T;H^1(\Omega))$,
\begin{multline}\label{bound_3}
\int_{t_0}^{t_1}\langle \partial_t u_i^{\eps,\delta}, \psi \rangle_{(H^1(\Omega))', H^1(\Omega)} \, dt + d_i\int_{t_0}^{t_1}\int_{\Omega}\frac{m_i(u_i^{\eps,\delta})^{m_i-1}}{1+\delta m_i(u_i^{\eps,\delta})^{m_i-1}}\na u_i^{\eps,\delta}\cdot \na \psi \, dx \, dt\\
+ d_i\delta\int_{t_0}^{t_1}\int_{\Omega}\na u_i^{\eps,\delta}\cdot \na \psi \, dx \, dt = \int_{t_0}^{t_1}\int_{\Omega}f_i^{\eps}(u^{\eps,\delta})\psi \, dx \, dt.
\end{multline}
The convergence of the first term on the left hand side and the last term on the right hand side of \eqref{bound_3} is immediate. From
\begin{equation*}
	\na u_i^{\eps,\delta} =  \frac{2}{m_i+p-1} \, (u_i^{\eps,\delta})^{\frac{3-m_i-p}{2}} \; \na (u_i^{\eps,\delta})^{\frac{m_i+p-1}{2}},
\end{equation*}
we can choose $1<p<3-m_i$ (since $m_i<2$) to get 
\begin{equation}\label{bound_4}
\|\nabla u_i^{\eps,\delta}\|_{L^2(Q_T)} \leq C_{\eps,T}
\end{equation}
thanks to the uniform bounds \eqref{bound_1} and \eqref{bound_2}.
For the third term on the left hand side of \eqref{bound_3}, we estimate
\begin{equation*}
\left|\delta \int_{t_0}^{t_1}\int_{\Omega}\na u_i^{\eps,\delta}\na \psi \, dx \, dt\right| \leq \delta\|\na u_i^{\eps,\delta}\|_{L^2(Q_T)}\|\na \psi\|_{L^2(Q_T)} \leq C_{\eps,T}{\delta}\|\na\psi\|_{L^2(Q_T)},
\end{equation*}
and, therefore, it converges to zero as $\delta \to 0$. Finally, the convergence 
\begin{equation*}
\int_{t_0}^{t_1}\int_{\Omega}\frac{m_i(u_i^{\eps,\delta})^{m_i-1}}{1+\delta m_i(u_i^{\eps,\delta})^{m_i-1}}\na u_i^{\eps,\delta}\cdot \na \psi \, dx \, dt \xrightarrow{\delta\to 0} \int_{t_0}^{t_1}\int_{\Omega}m_i(u_i^{\eps})^{m_i-1}\na u_i^{\eps}\cdot \na \psi \, dx \, dt 
\end{equation*}
follows from $\frac{m_i(u_i^{\eps,\delta})^{m_i-1}}{1+\delta m_i(u_i^{\eps,\delta})^{m_i-1}} \xrightarrow{\delta \to 0} m_i(\ue_i)^{m_i-1}$ a.e.\ in $Q_T$, \eqref{bound_1}, and \eqref{bound_4}.

\section{Proof of a uniform bound on approximate solutions}\label{app2} 
We provide a proof of Lemma \ref{add_lem0}. Define the new variables
\begin{equation*}
\we_i = \ue_i \log \ue_i + (\mu_i-1)\ue_i + e^{-\mu_i} \geq 0
\end{equation*}
and
\begin{equation*}
\ze_i = (\log \ue_i + \mu_i)(\ue_i)^{m_i} - \frac{(\ue_i)^{m_i} - e^{-\mu_i m_i}}{m_i} \geq 0,
\end{equation*}
where the non-negativity follows from elementary calculus. Direct computations give
\begin{equation*}
\partial_t \we_i - d_i\Delta \ze_i = (\log \ue_i + \mu_i)f^{\eps}_i(\ue) - d_im_i(\ue_i)^{m_i-2}|\na \ue_i|^2.
\end{equation*}
Therefore, thanks to \eqref{entropy-inequality},
\[
\partial_t\bra{\sum_{i=1}^{I}\we_i} - \Delta\bra{\sum_{i=1}^{I}d_i\ze_i} \leq C \, \sumi \bra{1+\ue_i\log \ue_i} \leq A + \beta \sumi \we_i
\]
with fixed positive constants $A, \beta > 0$ and a generic constant $C > 0$. Introducing $v(x, t) \coleq e^{-\beta t} \sumi \we_i$, this leads to 
\begin{equation}\label{c1}
\partial_t v - \Delta \bra{e^{-\beta t} \sumi d_i \ze_i} \leq A e^{-\beta t}.
\end{equation}
We integrate \eqref{c1} in time on $(0,t)$, then multiply the resultant by $e^{-\beta t} \sum_{i=1}^{I}d_i\ze_i$ and integrate in $Q_T$, to get
\begin{multline}\label{c2}
\int_{Q_T} v \, e^{-\beta t} \sumi d_i\ze_i\, dx \, dt - \int_{Q_T}\bra{\Delta \int_0^t e^{-\beta s} \sumi d_i\ze_i \, ds}\bra{e^{-\beta t} \sumi d_i \ze_i} dx \, dt\\
\leq \int_{Q_T} \bra{\frac{A}{\beta} + v(x,0)} e^{-\beta t} \sumi d_i\ze_i\, dx \, dt.
\end{multline}
For the second term on the left hand side of \eqref{c2}, we have
\begin{align*}
&- \int_{Q_T}\bra{\Delta \int_0^t e^{-\beta s} \sumi d_i\ze_i \, ds}\bra{e^{-\beta t} \sumi d_i \ze_i} dx \, dt\\
&\qquad = \int_{Q_T}\na \Bigg( \underbrace{\int_0^t e^{-\beta s} \sumi d_i\ze_i(x,s) \, ds}_{=:\,\varphi(x,t)} \Bigg) \cdot \na \Bigg( \underbrace{e^{-\beta t} \sumi d_i \ze_i(x,t)}_{=\partial_t\varphi(x,t)} \Bigg) \, dx \, dt\\
&\qquad = \frac 12 \int_0^T \frac{d}{dt}\int_{\Omega}|\na \varphi (x,t)|^2\, dx \, dt = \frac 12\int_{\Omega}|\na \varphi(x,T)|^2\, dx \geq 0
\end{align*}
where we used $\varphi(x,0) = 0$ at the last step. On the other hand, \eqref{c1} ensures that
\begin{equation}\label{aa}
-\Delta \int_0^T e^{-\beta t} \sumi d_i\ze_i \, dt \leq \frac{A}{\beta} + v(x,0).
\end{equation}
By putting $y = \int_0^T e^{-\beta t} \sum_{i=1}^I d_iz_i^{\eps} \, dt$, we have $-\Delta y + y\leq y + \frac{A}{\beta} + v(x,0)$. Multiplying both sides by $y$ and integrating over $\Omega$ results in 
\begin{equation*}
\|y\|_{H^1(\Omega)}^2 \leq 2 \|y\|_{L^2(\Omega)}^2 + \left\| \frac{A}{\beta} + v(0,x) \right\|_{L^2(\Omega)}^2.
\end{equation*}
By an interpolation estimate and Young's inequality, there exist $\theta \in (0,1)$ and $c, \gamma, \Gamma > 0$ with $\gamma$ arbitrarily small satisfying $\|y\|_{L^2(\Omega)}^2 \leq c \|y\|_{H^1(\Omega)}^{2\theta} \|y\|_{L^1(\Omega)}^{2(1-\theta)} \leq \gamma \|y\|_{H^1(\Omega)}^2 + \Gamma \|y\|_{L^1(\Omega)}^2$. Hence, we obtain 
\begin{equation*}
	\|y\|_{L^2(\Omega)} \leq C \bra{1 + \|y\|_{L^1(\Omega)}}.
\end{equation*}
Using \eqref{c2} and the above bound on $y$, it follows that
\begin{multline}\label{c4}
\int_{Q_T}\bra{\sumi w_i^\eps}\bra{\sumi d_iz_i^\eps}dx \, dt \leq C_T \int_\Omega \bra{\frac{A}{\beta} + v(x,0)} y \, dx \\
\qquad \leq C_T \bra{1 + \|y\|_{L^1(\Omega)}} \leq C_T \bra{1 + \int_{Q_T}\sumi d_iz_i^\eps \, dx \, dt}.
\end{multline}
From the definition of $w_i^\eps$ and $z_i^\eps$, it is easy to see that for any $\delta>0$, there exists some $C_\delta>0$ such that
\begin{equation*}
\sumi d_iz_i^\eps \leq \delta \bra{\sumi w_i^\eps}\bra{\sumi d_iz_i^\eps} + C_{\delta}.
\end{equation*}
We then derive from \eqref{c4} that
\begin{equation*}
\int_{Q_T}\bra{\sumi w_i^\eps}\bra{\sumi d_iz_i^\eps} dx \, dt \leq C.
\end{equation*}
The desired estimate \eqref{uniform-bound} is now an immediate consequence.

\medskip
\noindent{\bf Acknowledgements.} A major part of this work was carried out when B.\ Q.\ Tang visited the Institute of Science and Technology Austria (IST Austria). The hospitality of IST Austria is greatly acknowledged.

This work is partially supported by the International Research Training Group IGDK 1754 ``Optimization and Numerical Analysis for Partial Differential Equations with Nonsmooth Structures'', funded by the German Research Council (DFG) project number 188264188/GRK1754 and the Austrian Science Fund (FWF) under grant number W 1244-N18 and NAWI Graz.

\bibliographystyle{abbrv}
\bibliography{porous_medium}
\end{document}